\documentclass[11pt]{amsart}
\usepackage{amssymb}
\usepackage{latexsym} 
\usepackage{srcltx} 
\usepackage[normalem]{ulem}
\usepackage{esint}
\usepackage{xcolor}
\usepackage{cancel}
\usepackage{upgreek}
\usepackage{mathrsfs}
\usepackage[numbers]{natbib}   
\usepackage[colorlinks=true, citecolor=green]{hyperref}

\DeclareFontFamily{U}{mathx}{}
\DeclareFontShape{U}{mathx}{m}{n}{<-> mathx10}{}
\DeclareSymbolFont{mathx}{U}{mathx}{m}{n}
\DeclareMathAccent{\widecheck}{0}{mathx}{"71}

\newcommand{\Polyn}[2]{\mathcal{P}_{#1}{(\R^n)}}

\newcommand{\Rcap}[2]{\mathrm{Cap}_{\mathbf{I}_{#1},#2}}

\newcommand{\PRcap}[2]{\mathrm{Cap}_{\mathbb{I}_{#1},#2}}
\newcommand{\Dyadic}{\dot{\Delta}_j}
\newcommand{\R}{\mathbb{R}}
\newcommand{\LpNorm}[3]{\|#1\|_{L^{#2}{(#3)}}}
\newcommand{\Chi}[2]{\upchi_{#1}(#2)}
\newcommand{\Heat}{\mathsf{H}}

\newcommand{\Ms}[3]{ \mathsf{M}^{#1,#2}_{#3}}
\newcommand{\TrMs}[5]
{\dot{F}^{#1}_{#5}\Ms{#2}{#3}{#4}}
\newcommand{\Scharwzd}[1]{\mathcal{S}'_{#1}(\R^n)}
\newcommand{\Qs}{\Scharwzd{}/ \Polyn{}}

\newcommand{\PR}[1]{\mathbb{I}_{#1}}
\newtheorem{theorem}{Theorem}[section]
\newtheorem{lemma}[theorem]{Lemma}
\newtheorem{corollary}[theorem]{Corollary}
\newtheorem{proposition}[theorem]{Proposition}
\newtheorem{definition}[theorem]{Definition}
\newtheorem{remark}[theorem]{Remark}        
\numberwithin{equation}{section}

\begin{document}

\title[]{ Optimal global well-posedness for  semilinear heat equations}

\author[Gael. Y. Diebou ]{Gael Y. Diebou}
\address{Department of Mathematics,  University of Toronto, 40 St. George Street, Toronto, Canada.}
\email{gaely.diebou@utoronto.ca}

\author[Nguyen Cong Phuc]{Nguyen Cong Phuc}
\address{Department of Mathematics, Louisiana State University, 303 Lockett Hall, Baton Rouge, LA 70803, USA.}
\email{pcnguyen@math.lsu.edu}


\begin{abstract}    We address the question of global in time existence for the semilinear parabolic equation
\begin{equation*}
 \begin{cases}
 \partial_tu-\Delta u=|u|^{b-1}u \quad {\rm ~in~} \mathbb{R}^n\times (0,\infty),\\
 u(\cdot,0)= u_0 \quad {\rm ~in~} \mathbb{R}^n,   
\end{cases}
\end{equation*}
where $b> 1+2/n, n\geq 1$. We found a space of initial  data $Z$ that guarantees the existence of a global in time solution whenever the norm of the initial datum $u_0$ in $Z$ is small. The space $Z$ is optimal in the sense that  if $u_0$ is a nonnegative distribution and the equation admits a nonnegative solution $u\in L^b_{\rm loc}(\R^n\times(0,\infty))$, then $u_0\in Z$.
Moreover, our result contains all of the known results in the literature. In particular,  it implies   that global in time existence does  hold for end-point Besov and Triebel-Lizorkin spaces, which resolves an open problem posed by L. Brandolese and F. Cortez  (J. Funct. Anal. {\bf 276} (2019),  2589--2604).   
\end{abstract}

\maketitle

\section{Introduction}

One of the most studied nonlinear partial differential equations of parabolic type over the last six decades  is the semilinear Cauchy problem
 \begin{align}
 \begin{cases}\label{Cauchy-T}
 \partial_tu-\Delta u=|u|^{b-1}u \quad {\rm ~in~} \R^n\times (0,T),\\
 u(\cdot,0)= u_0 \quad {\rm ~in~} \R^n,
 \end{cases}   
\end{align}
where $n\in \mathbb{N}$, $b>1$, and $0<T\leq \infty$. In particular, many local and global in time well-posedness results for the Cauchy problem \eqref{Cauchy-T} were obtained in \cite{Wei80, Wei81, BP, Giga, KY, BC,  Mey, MYZ, Sou,  Q-H}, etc.
On the other hand, it follows from the earlier work of Fujita \cite{Fuj} that no nontrivial nonnegative global in time ($T=\infty$) solution   exists when $b<1+ 2/n$.  Such a nonexistence phenomenon also happens at the endpoint 
$b=1+ 2/n$ (see \cite{Hay} and \cite{Wei80, Wei81}).

In this paper, we are interested in the Cauchy problem \eqref{Cauchy-T} with singular initial data $u_0$.  For initial data in the Lebesgue space $L^a$,  $a \geq 1$, the number $n(b-1)/2$ is a critical threshold. It it known that if $1\leq a< n(b-1)/2$,  then no  nonnegative regular solution $u\in C([0,T], L^a(\R^n))$
exists in general for nonnegative initial data $u_0\in L^a(\R^n)$ (see \cite[Theorem 1]{Wei80}).

On the other hand, let  $b>1$ and $a\geq n(b-1)/2, a>1$ (or $a> n(b-1)/2, a\geq 1)$. In  \cite{Wei80},  Weissler used a fixed point theorem in a subspace of $C([0,T], L^a(\R^n))$   to construct a solution $u\in C([0,T], L^a(\R^n))\cap L^\infty_{\rm loc}((0, T], L^\infty(\R^n))$, $T=T(u_0)>0$, along with some other properties, to the Cauchy problem \eqref{Cauchy-T}.
Such solutions are sometimes referred to  as Weissler's solutions.

Later, Brezis and Cazenave \cite{BC} showed that this solution is unique in the class $C([0,T], L^a(\R^n)) \cap L^\infty_{\rm loc}((0, T], L^\infty(\R^n))$. Moreover, if $a\geq n(b-1)/2$ and $a>b$, then uniqueness also holds in the larger class 
 $C([0,T], L^a(\R^n))$ (see \cite{Wei80} for the case $a> n(b-1)/2$ and $a>b$ and \cite{BC} for the case $a= n(b-1)/2$ and $a>b$). However, we mention that uniqueness is lost in the class $C([0,T], L^a(\R^n))$
 in the critical case $a= n(b-1)/2$ and $a=b$, i.e., $a=b=n/(n-2)$ (see \cite{Ter} and also \cite{NS} for the case of a bounded domain).

It is  easy to verify that if $u(x,t)$ is a solution of  \eqref{Cauchy-T}, then for any $\lambda >0$, the function $u_\lambda(x,t):=\lambda^{\frac{2}{b-1}} u(\lambda x, \lambda^2 t)$ solves the same Cauchy problem  on $\R^n\times [0, \lambda^{-2} T]$ with initial 
condition $u_{0 \lambda}(x)=\lambda^{\frac{2}{b-1}} u_0(\lambda x)$.  Note that in the case $b=3$  this scale invariance is similar to that of the   Navier-Stokes system.  In \cite{Mey} (see also \cite{Mey01}),   Y. Meyer treated the  case $b=n=3$ as a simple model for  the 
three-dimensional Navier-Stokes system. In doing so, he posed the following conjecture  regarding \eqref{Cauchy-T} with $b=n=3$ (see \cite[Conjecture 2.1]{Mey}).

\smallskip

\noindent {\bf Conjecture:} There exists a Banach space $(\mathcal{Z}, \|\cdot\|_{\mathcal{Z}})$ such that the following holds.

\noindent  $(i)$ $\|f\|_{L^4(\R^3)}^2 \lesssim \| \nabla f\|_{L^2(\R^3)} \| f \|_{\mathcal{Z}}$.

\noindent  $(ii)$ If $u_0\in C^\infty_c(\R^3)$ is such that $\| u_0 \|_{\mathcal{Z}}\leq\eta$, the lifetime $T$ of  the corresponding solution of  \eqref{Cauchy-T} with $b=n=3$ is infinite. Here, $\eta$ is a  (small) positive   constant.  

\noindent  $(iii)$  $\| \lambda f(\lambda \cdot)\|_{\mathcal{Z}} = \|f\|_{\mathcal{Z}}$ for all $\lambda>0$.

\noindent  $(iv)$ $\|f\|_{\mathcal{Z}}$ is small when $f$ is oscillating.\\

Throughout the manuscript, for nonnegative numbers $N,M$, we write $N\lesssim M$ to mean $N\leq CM$ for an unimportant constant $C>0$. 
\smallskip

Let $\dot{B}^{\alpha}_{p, q}(\R^n)$  (resp. $\dot{F}^{\alpha}_{p, q}(\R^n)$), $\alpha \in \R, p, q\in(0,\infty]$,  denote the Besov spaces (resp. Triebel-Lizorkin spaces) in $\R^n$ (see Section \ref{SecFunctSpace}).
Y. Meyer observed that this conjecture holds for $\mathcal{Z}= \dot{B}^{-1/2}_{6, \infty}(\R^3)$ and even more generally for $\mathcal{Z}=\dot{B}^{-1 +3/p}_{p, \infty}(\R^3), 3<p<9$, (see \cite{BCo, KY, MYZ}).
Moreover, in these cases, property $(ii)$ holds also for any $u_0\in L^3(\R^3)$ and the lifetime of the corresponding  Weissler's solution is infinite.
Note that $L^3(\R^3)\subset \dot{B}^{-1 +3/p}_{p, \infty}(\R^3) \subset \dot{B}^{-1}_{\infty, \infty}(\R^3)$, $3<p\leq \infty$. The significance of Besov spaces (or Triebel-Lizorkin spaces) of {\it negative} order 
 lies in the fact that the condition of  being sufficiently small in such spaces can be verified for highly oscillating data, which justifies property $(iv)$. 
  For example, if $f\in L^3(\R^3)$, with arbitrarily large norm,  and  $f_k(x)=e^{i k\cdot x} f(x), k\in\mathbb Z^3$, then $\|f_k\|_{L^3(\R^3)}=\|f\|_{L^3(\R^3)}$, whereas $\|f_k\|_{\dot{B}^{-1 +3/p}_{p, \infty}(\R^3)}\rightarrow 0$ as $|k|\rightarrow\infty$, provided $3<p\leq \infty$ (see \cite[Section 4.1]{Can}).

The main point of Meyer's conjecture is to find the largest Banach space $\mathcal{Z}$ such that properties  $(i)$-$(iv)$ holds. 
 Meyer \cite{Mey} proved property $(i)$ for $\mathcal{Z}=\dot{B}^{-1}_{\infty, \infty}(\R^3)$ and raised the question whether or not the maximal Besov space 
$\dot{B}^{-1}_{\infty, \infty}(\R^3)$ fullfils $(ii)$.  This question was addressed by L. Brandolese and F. Cortez in \cite[Theorem 2.2]{BCo}, where the authors showed that for any $\eta>0$ and $3<q\leq \infty$, there exists $u_0\in \mathcal{S}(\R^3)$ such that 
$$\|u_0\|_{\dot{B}^{-2/3}_{9, q}(\R^3)} \leq \eta,$$
and the lifetime $T$ of the corresponding Weissler's  solution   is finite. Thus,  the conjecture above fails for  $\mathcal{Z}=\dot{B}^{-2/3}_{9, q}(\R^3), 3<q\leq \infty$, and hence also for $\mathcal{Z}=\dot{B}^{-1}_{\infty, \infty}(\R^3)$  (at least 
 when the initial data are assumed to be in $\mathcal{S}(\R^3)$ instead of in  $C^\infty_c(\R^3)$ as stated in $(ii)$).  More generally, the existence of finite time blowup Weissler's solutions was also obtained in \cite[Theorem 2.2]{BCo} for general $n$ and $ b \in \mathbb{N}$,  $b>1+2/n$, and for initial datum  $u_0\in \mathcal{S}(\R^n)$ such that the Besov norm $\|u_0\|_{\dot{B}^{-2/b}_{nb(b-1)/2, q}}$ for $q\in (b,\infty]$, is small. The endpoint case $q=b$ was stated as an open problem in \cite[page 12]{BCo}:
 \emph{ Let $b>1+2/n$ and $u_0$ a smooth and well decaying initial datum
as $|x|\rightarrow +\infty$. Does the smallness assumption  $\|u_0\|_{\dot{B}^{-2/b}_{nb(b-1)/2, b}}\ll1$
 imply that the solution of the Cauchy problem  \eqref{Cauchy-T} is global in time?}

One of the goals of this paper is to provide a positive answer to the above question, even for initial data in $L^a(\R^n)$ for any $a\geq n(b-1)/2$ and the corresponding Weissler's solutions. At the same time, we also obtain a  space, optimal in some sense, that resolves the conjecture of Y. Meyer.   In fact, all known admissible spaces of initial data for \eqref{Cauchy-T} to have  global in time solutions are included in our space. More precisely, our main results are stated in Theorems  \ref{thm:maintheorem} and \ref{PerRes} below and our space of initial data is defined by 
$$Z_{b, b'}(\R^n):=\left\{f\in \mathcal{S}'(\R^n): \|f\|_{Z_{b,b'}(\R^n)}<+\infty \right\},$$ 
where $b> 1+ 2/n $, $b'=b/(b-1)$, and 
\begin{align*}
\|f\|_{Z_{b,b'}(\R^n)}:= \sup_{E\subset \R^{n+1}_+} \left(\frac{ \iint_{E}| e^{t\Delta} f(x) |^b dx dt}{{\rm cap}_{2,1, b'}(E)}  \right)^{1/b}.
\end{align*}
Here the supremum is taken over all compact subsets $E$ of $\R^{n+1}_{+}=\R^n\times (0, \infty)$, and ${\rm cap}_{2,1, b'}(E)$ is a parabolic capacity defined by
$${\rm cap}_{2,1, b'}(E) =\inf_{\Phi\in K(E)} \left\{ \int_{\R^{n+1}} (|\partial_t \Phi|^{b'} + |D^2 \Phi|^{b'}) dx dt \right\}, $$
with  
$$K(E)=\{\Phi\in C_c^\infty(\R^{n+1}), 0\leq \Phi\leq 1,   \Phi = 1 {\rm ~in ~a ~neighborhood ~of~} E\}.$$
See Section \ref{SecFunctSpace} for different notions of capacities used in this paper. 
Of course, when dealing with Meyer's conjecture, we take $\mathcal{Z}=Z_{3, 3/2}(\R^3)$.  

In particular, the following theorem is a direct consequence of Theorems  \ref{thm:maintheorem}, \ref{PerRes}, and \ref{BP-Hung} below.
 \begin{theorem}\label{ForIntro}
Let $b>1+ \frac 2 n$ and $u_0\in \mathcal{S}'(\R^n)$. There exists $\varepsilon=\varepsilon(n,b)>0$ such that  if $ u_0\in Z_{{b, b'}}(\R^n)$ with
 $\|u_0\|_{Z_{b, b'}(\R^n)} \leq \varepsilon,$ then there exists a   solution $u$ to the Cauchy problem \eqref{Cauchy-T} with $T=\infty$. Moreover, 
 if, in addition, $u\in L^a(\R^n)$, 
 $a\geq \frac{n(b-1)}{2}$, then $u\in C([0,\infty), L^a(\R^n))$, and $u$ coincides wih the corresponding Weissler's solution. 
 
 Conversely, if $u$ is a nonnegative distribution in $\R^n$ and \eqref{Cauchy-T} admits a nonnegative global in time solution $u\in L^b_{\rm loc}(\R^n\times (0,\infty))$,  then $ u_0\in Z_{{b, b'}}(\R^n)$.
\end{theorem}

\begin{remark}
    Interestingly,  one can show  that $Z_{b, b'}(\R^n)$ is related to a class of Koch-Tataru type spaces in the following  way:
$$T_{b,\epsilon}(\R^n) \hookrightarrow Z_{b, b'}(\R^n) \hookrightarrow T_{b,0}(\R^n), \qquad \epsilon>0,$$
where for $b>1+2/n$ and $\epsilon\geq 0$,
$$T_{b, \epsilon}(\R^n):=\left\{f\in \mathcal{S}'(\R^n): \|f\|_{T_{b,\epsilon}(\R^n)}<+\infty \right\},$$
with
$$\|f\|_{T_{b,\epsilon}(\R^n)}:= \left(\sup_{x\in\R^n,\,  r>0}  r^{2 b' (1+\epsilon)-n-2} \int_0^{r^2} \int_{B_r(x)} |e^{t\Delta} f|^{b(1+\epsilon)} dy dt \right)^{\frac{1}{b(1+\epsilon)}}.$$
For this see Theorem \ref{KT-FP} below. Note that the endpoint space $T_{b,0}(\R^n)$ is an analogue of the Koch-Tataru space (see \cite{KT})  in terms of  scaling and optimal integrability. Indeed, in order to make sense of a solution $u$ to \eqref{Cauchy-T} we need at least that $u\in L^b_{\rm loc}(\R^n\times [0, \infty))$. Moreover, since the equation is scale and translation invariance, we want    that 
$$\sup_{x\in \R^n, r>0} r^{2b' -n-2} \int_0^{r^2} \int_{B_r(x)} |u|^b dy dt <+\infty.  $$
Thus, the space of initial data should  at least  be included in  $T_{b, 0}(\R^n)$. 

\end{remark}
 
Theorem \ref{ForIntro} shows that the space  $Z_{b, b'}(\R^n)$ is sharp for the existence of global in time solutions to the Cauchy problem \eqref{Cauchy-T} at least when dealing with nonnegative initial data and nonnegative solutions. See Theorem \ref{BP-Hung} below for more results in the case of nonnegative initial data. In particular, 
it  provides explicit characterizations of a (nonlinear) parabolic weighted norm inequality intrinsic to the Cauchy problem \eqref{Cauchy-T} discovered earlier in \cite{BP} (see also \cite{KV}).

For initial tempered distributional data, as a consequence of Theorem \ref{BbFbZ} below, one has  
  $$\dot{B}^{-2/b}_{nb(b-1)/2, b}(\R^n)\hookrightarrow \dot{F}^{-2/b}_{nb(b-1)/2, b}(\R^n) \hookrightarrow Z_{b, b'}(\R^n), \quad b>1+ \frac 2 n,$$
 which provides the   answer to  the question of Brandolese and Cortez not only for  the Besov space  $\dot{B}^{-2/b}_{nb(b-1)/2, b}(\R^n)$ but also for the larger  Triebel-Lizorkin space  $\dot{F}^{-2/b}_{nb(b-1)/2, b}(\R^n)$.
  These spaces are not, in general, comparable to the Besov-Morrey space 
   $\dot{B}_{\infty}^{\frac{n}{p}-\frac{2}{b-1}} \mathcal{M}^{q, \frac{n q }{p}}$ for any 
  $b\leq q\leq p$  and $n(b-1)/2<p< n b(b-1)/2$ that was used  by Kozono and Yamayaki in \cite{KY} (see Section \ref{SecFunctSpace} for the definition of Besov-Morrey spaces).
   In  Corollary \ref{KYinZ} below, we show  that in fact these Besov-Morrey  spaces are also included  in $Z_{b,b'}(\R^n)$.
  
  The smallness of the norm of $u_0$ in the Morrey space  $\mathcal{M}^{q, \frac{2 q}{b-1}}(\R^n)$, $1<q\leq n(b-1)/2$, also yields a global in time solution to the Cauchy problem \eqref{Cauchy-T} (see \cite{Sou}). 
More generally, Q.-H. Nguyen \cite{Q-H} obtained a similar existence result  for $u_0$ belonging to the space of signed measures $\mu$ on $\R^n$ such that     
 $$|\mu|(K)\leq A\,  {\rm Cap}_{ {\bf I}_{\frac{2}{b}, b'}}(K)$$
 for all compact sets $K\subset\R^n$.  Here, ${\rm Cap}_{ {\bf I}_{\frac{2}{b}, b'}}(\cdot)$ is a Riesz capacity of order $2/b$ (see Section \ref{SecFunctSpace}). These results are also special cases of our result in which one can take 
 initial data in $\TrMs{-\frac{2}{b}}{\frac{2}{b}}{b'}{b}{b}(\R^n)$, a Triebel-Lizorkin type space based on a Sobolev multipler type space (see Section \ref{SecFunctSpace} and Theorem \ref{BbFbZ} and Corollary \ref{MorreyandCapspace}). 
 
Finally, we remark that, by Theorem \ref{thm:cap-equiv} and Proposition \ref{prop:equiv} below, a tempered distribution $f$ belongs to $Z_{b,b'}(\R^n)$, $b>1+ 2/n$, if and only if there exists a constant $C>0$ such that 
\begin{equation}\label{I2abstract}
\mathbb{I}_{2} [\mathbb{I}_2 [\nu]^{b'}] \lesssim \mathbb{I}_2[\nu] <+\infty \quad {\rm a.e.},
\end{equation}  
 where $\nu(x,t)= \upchi_{(0,\infty)}(t) |e^{t\Delta} f(x)|^b$. Here  $\mathbb{I}_2$ is the  parabolic Riesz potential of order $2$ defined  by 
$$\mathbb{I}_{2}[\nu](x,t)=\int_{\R^{n+1}} \frac{\nu(y,s)}{(|t-s| +|x-y|^2)^{n/2}} dyds,$$ and $\upchi_{I}$ denotes the characteristic function of the set $I$.
 
 Then by the theory developed in \cite{KV} (see Theorem 4.9 and Proposition 2.9 of \cite{KV}), condition \eqref{I2abstract} is also  equivalent to a more abstract condition: {\it There exist a function $v\geq \nu$
 on $\R^{n+1}$ and a constant $B>0$ such that $\mathbb{I}_{2}[v^{b}]\leq B v$.} Thus in the quadratic case $b=2$, the space $Z_{2,2}(\R^n), n\geq 3$,  coincides with the space $\mathcal{V}^{\alpha, \beta}$, with $\alpha=2$ and $\beta=0$, that was employed by P. G. Lemari\'e-Rieusset in \cite{LR18} to deal with the equation $\partial_t u + (-\Delta)^{\alpha/2} u= (-\Delta)^{\beta/2} u^2$ in the case $0<\alpha< n+2\beta$ and $0<\beta<\alpha$.

\section{Capacities and function spaces}\label{SecFunctSpace}

\subsection{ Parabolic capacities and potentials}
In what follows,  for a Borel set $E\subset\R^d, d\in \mathbb N$, we denote by $\mathfrak{M}(E)$ the set of signed Radon measures in $E$ and its positive cone is denoted by $\mathfrak{M}^+(E)$.

 Given  $\mu\in \mathfrak{M}^+(\R^n), n\in \mathbb N$, we define its Riesz potential of order $\alpha\in (0,n)$ by 
 \begin{equation*}
\mathbf{I}_\alpha[\mu](x)=\int_{\R^n} \frac{1}{|x-y|^{n-\alpha}} d\mu(y) ,\quad x\in \R^n.   
 \end{equation*} 
Similarly, for $\nu\in \mathfrak{M}^+(\R^{n+1})$, we define its parabolic Riesz potential of order $\beta \in (0,n+2)$ by 
$$\mathbb{I}_{\beta}[\nu](x,t)=\int_{\R^{n+1}} \frac{1}{(|t-s| +|x-y|^2)^{(n+2-\beta)/2}} d\nu(y,s),$$ 
where $x\in\R^n, \, t\in\R$. 

Denote by $B_r(x)$ the open ball in $\R^n$ centered at $x$ with radius $r>0$ and    
$Q_r(x,t):=B_r(x)\times (t-r^2/2,t+r^2/2)$ the parabolic cylinder.
It is not hard to see that 
 \begin{equation*}
\mathbf{I}_\alpha[\mu](x)= (n-\alpha) \int_0^\infty\frac{\mu(B_r(x))}{r^{n-\alpha}} \frac{dr}{r},\quad x\in \R^n,  
 \end{equation*}
and 
\begin{equation*}
\PR{\beta}[\nu](x,t)\simeq \int_0^\infty  \frac{\nu(Q_r(x,t))}{r^{n+2-\beta}}\frac{dr}{r},\quad (x,t)\in \R^{n+1}.
\end{equation*}
Recall that  the heat kernel  on $\R^{n+1}$ is given by
\begin{equation*}
\Heat(x,t)=  \frac{1}{(4\pi t)^{\frac{n}{2}}} e^{-\frac{|x|^2}{4t}} \Chi{(0,\infty)}{t}. 
\end{equation*}
 We also define the   heat potential of a measure $\nu\in \mathfrak{M}^+(\R^{n+1})$ by
\begin{equation*}
\Heat[\nu](x,t)= [\Heat\ast \nu](x,t)=\int_{\R^{n+1}}\Heat(x-y,t-s) d \nu(y,s)   
\end{equation*}
for any $(x,t)\in \R^{n+1}$.

With these, we can now define, respectively, the Riesz, the parabolic Riesz, and the heat capacities as follows:
$$ \Rcap{\alpha}{p}(K)=\inf\left \{\LpNorm{f}{p}{\R^n}^p:\, f\geq 0,\, \, \mathbf{I}_{\alpha}[f]\geq 1\,\, \mathrm{on}\,\, K \right\}   $$
for any set $K\subset\R^n$, where $\alpha\in(0,n)$ and $1<p<\frac{n}{\alpha}$,
and 
$$ {\rm Cap}_{\mathbb{I}_\beta, p_1 } (E)=\inf \left\{\LpNorm{f}{p_1}{\R^{n+1}}^{p_1}:\, f\geq 0,\, \, \mathbb{I}_{\beta}[f]\geq 1\,\, \mathrm{on}\,\, E \right\},   $$
\begin{align*}
 \mathrm{Cap}_{\mathsf{H},  p_2}  (E)=\inf \left\{\LpNorm{f}{p_2}{\R^{n+1}}^{p_2}:\, f\geq 0,\, \, \Heat[f]\geq 1\,\, \mathrm{on}\,\, E \right \}  
\end{align*}
for any set $E\subset\R^{n+1}$, where $\beta\in (0,n+2)$, $1<p_1<\frac{n+2}{\beta}$, and $1<p_2<\frac{n+2}{2}$.

We next introduce  the following  variational  homogeneous parabolic capacities.  For $p>1$ and for each compact set $E\subset\R^{n+1}$, we define 
$${\rm cap}_{2,1, p}(E) =\inf \left\{ \int_{\R^{n+1}} (|\partial_t \Phi|^{p} + |D^2 \Phi|^{p}) dx dt \right\}, $$
where the infimum is taken over all functions $\Phi\in C_c^\infty(\R^{n+1}), 0\leq \Phi\leq 1$, and  $\Phi \geq 1$   in a neighborhood of $E$ (hence, $\Phi = 1$   in a neighborhood of $E$). Similarly,
 we define
\begin{align*}
{\rm c}_{2,1, p}(E) =\inf \left\{ \int_{\R^{n+1}} (|\partial_t \Phi|^{p} + |D^2 \Phi|^{p}) dx dt \right\}, 
\end{align*}
where the infimum is taken over all  $\Phi\in C_c^\infty(\R^{n+1})$ such that  $\Phi \geq 1$ on $E$.

We will need the following parabolic Hedberg type  lemma.
\begin{lemma}\label{Hed} Let $0<\alpha<\beta<n+2$. Then for any  $\mu\in \mathfrak{M}^+(\R^{n+1})$ we have 
\begin{align*}
\mathbb{I}_{\alpha}[\mu](x,t) \leq c\, \mathbb{M}_{\beta}\mu(x,t)^{\alpha/\beta} \mathbb{M} \mu(x,t)^{1-\alpha/\beta} \leq c\, \mathbb{I}_{\beta}[\mu](x,t)^{\alpha/\beta} \mathbb{M} \mu(x,t)^{1-\alpha/\beta}  
\end{align*}
for all $(x,t)\in \R^{n+1}$, where 
$$\mathbb{M}_{\beta}\mu(x,t)=\sup_{r>0}\frac{\mu(Q_r(x, t))}{r^{n+2-\beta}}, \qquad 0\leq \beta<n+2,$$ 
and  $\mathbb{M}\mu(x,t)=\mathbb{M}_{0}\mu(x,t)$ (the parabolic Hardy-Littlewood maximal function of $\mu$). 
\end{lemma}
\begin{proof} It is enough to prove the first inequality.  For any $\delta>0$, we have 
\begin{align*}
\mathbb{I}_{\alpha}[\mu](x,t)& \lesssim  \int_0^{\delta}\frac{\mu(Q_r(x,t))}{r^{n+2-\alpha}} \frac{dr}{r} +  \int_{\delta}^{\infty} \frac{\mu(Q_r(x,t))}{r^{n+2-\alpha}} \frac{dr}{r} \\
& \lesssim \mathbb{M} \mu(x,t) \delta^{\alpha} +  \mathbb{M}_{\beta}[\mu](x,t) \delta^{\alpha-\beta}.
\end{align*}
We may assume that both $\mathbb{M} \mu(x,t)$ and $\mathbb{M}_{\beta}[\mu](x,t)$ are finite. Then choosing 
$$\delta=\left[\mathbb{M}_{\beta}[\mu](x,t)/\mathbb{M} \mu(x,t)\right]^{\frac{1}{\beta}},$$
 we obtain the lemma.
\end{proof}

The following equivalence of parabolic capacities will be important for our development later.

\begin{theorem}\label{thm:cap-equiv} Let $1<p<\frac{n+2}{2}$. There exist positive constants $C_1, C_2, C_3$, and $C_4$ such that 
\begin{align}\label{4cap}
\ {\rm c}_{2,1, p}(E) &\leq C_1\, {\rm cap}_{2,1, p}(E) \leq C_2\, \PRcap{2}{p}(E)\\
& \leq C_3\,   \mathrm{Cap}_{\mathsf{H},  p}(E) \leq C_4 \, {\rm c}_{2,1, p}(E) \nonumber
\end{align}
for all compact sets $E\subset\R^{n+1}$.
\end{theorem}
\begin{proof}
The first inequality in \eqref{4cap} is trivial and the bound $$\PRcap{2}{p}(E) \leq c \,  \mathrm{Cap}_{\mathsf{H},  p}(E) $$ is also obvious  since 
$\Heat(x,t)\leq c \left(|t| +|x|^{2}\right )^{-n/2}$ on $\R^{n+1}\setminus(0,0)$.

To prove the last inequality, let  $\Phi\in C_c^\infty(\R^{n+1})$ be such that   $\Phi\geq 1$  on a compact set  $E\subset\R^{n+1}$. Note that  
$$ \Phi=\Heat*(\partial_t \Phi -\Delta \Phi) \leq \Heat*(|\partial_t \Phi -\Delta \Phi|),$$
and thus 
$$ \mathrm{Cap}_{\mathsf{H},  p}(E) \leq \int_{\R^{n+1}} |\partial_t \Phi -\Delta \Phi|^p dx dt \lesssim \int_{\R^{n+1}} (|\partial_t \Phi|^p + | D^2 \Phi|^p) dx dt.$$
Minimizing the right-hand side above over all such $\Phi$, we obtain  
$$ \mathrm{Cap}_{\mathsf{H},  p}(E) \lesssim \, {\rm c}_{2,1, p}(E).$$

It remains to show the second inequality in \eqref{4cap}.   Since $\PRcap{2}{p}$ is an outer capacity, it suffices to show that 
\begin{equation}\label{m-cap}
{\rm cap}_{2,1, p}(E) \lesssim\, \PRcap{2}{p}(O)
\end{equation}
for any bounded open set $O \supset E$, where $c$ is independent of $E$ and $O$. To this end, we shall use  the smooth truncation method for parabolic potentials 
(see \cite{Ad76} for the elliptic case; see also \cite{Maz}).

 Let  $f\in L^p(\R^{n+1}), f\geq 0$, be such that 
$\mathbb{I}_2[f] \geq 1$ on $O$, where $O$ is a bounded open set containing $E$.  Note that by H\"older's inequality and the fact that $1<p<\frac{n+2}{2}$, we have 
$$\int_{\R^{n+1}\setminus Q_{R}(0)} \frac{f(y,s)}{(|t-s|+|x-y|^2)^{n/2}} dyds\leq \frac12$$
 for all $(x,t)\in O$, provided $R>0$ is sufficiently large. 
  Thus by replacing $f$ with $2f\upchi_{Q_R(0)}$ for a sufficiently large $R$, we may assume that $f$ has compact support.

Let  $\mathbb{B}_r(z)$ denote the  Euclidean ball centered at $z$ with radius $r>0$ in $\R^{n+1}$.
 Let $\varphi$ be a standard mollifier in $\R^{n+1}$, i.e., $\varphi\in C_c^\infty(\mathbb{B}_{1}(0))$, $\varphi$ is nonnegative and radial, and 
$\int_{\R^{n+1}} \varphi dxdt=1$.  For each $\epsilon>0$, let $\varphi_{\epsilon}(x,t)=\frac{1}{\epsilon^{n+1}} \varphi(x/\epsilon,t/\epsilon)$.
Then, provided $\varepsilon$ is small, we have 
$$\varphi_\epsilon*\mathbb{I}_2[f]=\mathbb{I}_2[\varphi_\epsilon*f]\geq 1$$
in $O'$ for some bounded open set $O'$ such that $E\subset O'\subset \overline{O'}\subset O$.

Let $T\in C^\infty(\R, [0, \infty))$ be such that $T(s)=0$ for $s\leq \frac12$ and $T(s)=1$ for $s\geq 1$.   Set 
$$\Phi_\epsilon = T(\mathbb{I}_2[\varphi_\epsilon*f]).$$
Then we have 
$$\partial_t \Phi_\epsilon = T'(\mathbb{I}_2[\varphi_\epsilon*f]) \partial_t\mathbb{I}_2[\varphi_\epsilon*f]$$
and for any $i, j\in \{1,2, \dots, n\}$,
$$D_{ij} \Phi_\epsilon =T''(\mathbb{I}_2[\varphi_\epsilon*f])D_i\mathbb{I}_2[\varphi_\epsilon*f] D_j\mathbb{I}_2[\varphi_\epsilon*f] + T'(\mathbb{I}_2[\varphi_\epsilon*f])D_{ij}\mathbb{I}_2[\varphi_\epsilon*f].$$
Thus it follows from Lemma \ref{Hed} that 
\begin{align*}
|\partial_t \Phi_\epsilon| + |D_{ij} \Phi_\epsilon| &\lesssim |\partial_t\mathbb{I}_2[\varphi_\epsilon*f]|  +   |D_{ij}\mathbb{I}_2[\varphi_\epsilon*f]| + \mathbb{I}_2[\varphi_\epsilon*f]^{-1} \mathbb{I}_1[\varphi_\epsilon*f]^2\\
& \lesssim |\partial_t\mathbb{I}_2[\varphi_\epsilon*f]|  +  |D_{ij}\mathbb{I}_2[\varphi_\epsilon*f]| +   \mathbb{M}(\varphi_\epsilon*f),
\end{align*}
where in the first inequality we used the fact that $s T''(s)\leq c$ for all $s\geq 0$ and $|D_i\mathbb{I}_2[\varphi_\epsilon*f]| \lesssim |\mathbb{I}_1[\varphi_\epsilon*f]|$ for any $i\in\{1,2, \dots, n\}$.
On the other hand, observe that 
 $$\partial_t\mathbb{I}_2[\varphi_\epsilon*f]= K_0 *(\varphi_\epsilon*f)\quad {\rm and} \quad D_{ij}\mathbb{I}_2[\varphi_\epsilon*f] =K_1 *(\varphi_\epsilon*f)$$
are Calder\'on--Zygmund operators acting on $\varphi_\epsilon*f$.  The boundedness of  the operators with kernels $K_0$, $K_1$ and of $\mathbb{M}$ on $L^p(\R^{n+1})$ now yields
\begin{equation*}
\int_{\R^{n+1}} |\partial_t \Phi_\epsilon|^p + |D_{ij} \Phi_\epsilon|^p dxdt \lesssim \int_{\R^{n+1}} (\varphi_\epsilon* f)^p dxdt \lesssim \int_{\R^{n+1}} f^p dxdt
\end{equation*} 

Note that $0\leq \Phi_\epsilon\leq 1$, $\Phi_\epsilon = 1$ in $O'\supset E$, and  $\Phi_\epsilon\in C_c^\infty(\R^{n+1})$ since $f$ has compact support.
We thus have 
\begin{equation*}
{\rm cap}_{2,1, p}(E) \leq  \int_{\R^{n+1}} |\partial_t \Phi_\epsilon|^p + |D^2 \Phi_\epsilon|^p dxdt \lesssim \int_{\R^{n+1}}f^p dxdt.
\end{equation*}
Minimizing the right-hand side over all such $f$, we obtain   \eqref{m-cap} as desired.
\end{proof}

\begin{remark}
The equivalence of $\PRcap{2}{p}$ and $\mathrm{Cap}_{\mathsf{H}, p}$ was obtained earlier by  Q.-H. Nguyen in \textnormal{\cite{Q-H}} (see Proposition 2.8 and the proof of Corollary 2.10 in \textnormal{\cite{Q-H})}.
We mention here that the proof given in \textnormal{\cite{Q-H}} is based on the following pointwise bounds
 \begin{align*}
 \, \int_0^\infty\frac{\nu(B_r(x)\times (t-r^2,t-\frac{r^2}{4}))}{r^{n}}\frac{dr}{r}\lesssim\Heat[\nu](x,t)\lesssim\,  \PR{2}[\nu](x,t)    
 \end{align*}
for any $\nu\in \mathfrak{M}^{+}(\R^{n+1})$ and $(x,t)\in\R^{n+1}$.
\end{remark}

We record the following important characterizations of parabolic trace inequality that can be found in \cite[Theorem 4.36]{Q-H};  see also \cite{MV}  for the elliptic case.
\begin{proposition}\label{prop:equiv}
Let $\nu\in \mathfrak{M}^+(\R^{n+1})$ and $1<p<\frac{n+2}{\beta}$. The following statements are equivalent:  
\begin{itemize}
    \item For every function $f\in L^p(\R^{n+1}), f\geq 0,$
    \begin{equation*}
    \int_{\R^{n+1}} \mathbb{I}_\beta[f]^{p} d\nu \leq C_1 \int_{\R^{n+1}} f^p dx dt 
    \end{equation*}
    \item For every compact set $E\subset \R^{n+1}$,
    \begin{equation*}
    \nu(E)\leq C_2 {\rm Cap}_{\mathbb{I}_\beta, p}(E).    
    \end{equation*}
\item For every compact set $E\subset \R^{n+1}$,
    \begin{equation*}
    \int_E \PR{\beta}[\nu]^{p'}dxdt\leq C_3 {\rm Cap}_{\mathbb{I}_\beta, p}(E), \quad p'=\frac{p}{p-1}.    
    \end{equation*}
\item  $\PR{\beta}[\nu]< +\infty$ a.e. and 
\begin{equation*}
\PR{\beta}[(\PR{\beta}[\nu])^{p'}] \leq   C_4\PR{\beta}[\nu]\quad \mbox{a.e.} 
\end{equation*}
\end{itemize}
Here  $C_1,C_2, C_3$, and $C_4$ are positive constants independent of $\nu$ and $E$. Moreover, 
$$C_1\simeq C_2\simeq (C_3)^{\frac{p-1}{p}} \simeq (C_4)^{p-1}.$$
\end{proposition}

\subsection{Triebel-Lizorkin and Besov type spaces}
Let $\upchi \in C^{\infty}_{c}(\R^n)$, $0\leq \upchi\leq 1$,  be such that $ \upchi(\xi)=1$ for $|\xi|\leq 1$ and $\upchi(\xi)=0$ for  $|\xi|\geq 2$. For $j\in \mathbb{Z}$, define 
\begin{equation}\label{varphij}
\varphi_{j}(\xi)=\varphi(2^{-j}\xi),  {\rm ~where~} 
\varphi(\xi)=\upchi(\xi)-\upchi(2\xi), 
 \end{equation}
 so that
\begin{align*}
\sum_{j\in \mathbb Z}\varphi_j(\xi)=1,\quad \forall \,\xi \neq 0.
\end{align*}

The homogeneous Littlewood-Paley projection  is the Fourier multiplier operator
\begin{equation*}
\Dyadic f   =\varphi_j(D)f=\mathcal{F}^{-1}(\varphi(2^{-j}\cdot)\mathcal{F}f)  =2^{n j}f*\mathcal{F}^{-1}(\varphi) (2^j \cdot) 
\end{equation*}
defined for  $f\in \mathcal{S}'(\R^n)$,
 where $\mathcal{F}$ and $\mathcal{F}^{-1}$ denote the Fourier and inverse Fourier transforms, respectively,  with $$\mathcal{F}h(\xi):=\int_{\R^n} h(x)e^{- 2\pi i x\cdot\xi} dx, \qquad \forall h\in \mathcal{S}(\R^n). $$ 
 
 Note that $\Dyadic f =\Dyadic (f+P)$ 
 for any polynomial $P$ in $\R^n$ and in general we have 
$$f= \sum_{j\in \mathbb{Z}} \Dyadic f \quad {\rm in~} \mathcal{S}'(\R^n)/\mathcal{P},$$
where $\mathcal{S}'(\R^n)/\mathcal{P}$ is the space of tempered distribitions modulo polynomials in $\R^n$ (see , e.g., \cite[Proposition 3.1]{LR}). On the other hand,  
 for any $N\in \mathbb{Z}$ and $f\in \mathcal{S}'(\R^n)$, we have 
$$f=\mathcal{F}^{-1}(\upchi(2^{-N}\cdot)\mathcal{F}f) + \sum_{j=N}^\infty \Dyadic f \quad {\rm in~} \mathcal{S}'(\R^n).$$
 Moreover, if $\lim_{N\rightarrow - \infty}\mathcal{F}^{-1}(\upchi(2^{-N}\cdot)\mathcal{F}f)=0$ in $\mathcal{S}'(\R^n)$, then it holds  (see, e.g.,  \cite[Theorem 3.1]{LR})
 that 
$$f= \sum_{j\in \mathbb{Z}} \Dyadic f \quad {\rm in~} \mathcal{S}'(\R^n).$$

\begin{definition}
Let $s\in \R$ and $p, q \in (0,\infty]$. Define
\begin{align*}
 \|f\|_{\dot{F}^s_{p, q}(\R^n)}:= \big\|\big(\sum_{j\in \mathbb{Z}}\big[2^{js}\Dyadic f\big]^q\big)^{1/q}\big\|_{L^p},\quad q\in (0, \infty], \, p\in (0,\infty)\footnotemark[1],  \\
 \|f\|_{\dot{B}^s_{p, q}(\R^n)}:= \big(\sum_{j\in \mathbb{Z}}\big[2^{js}\|\Dyadic f\|_{L^p}\big]^q\big)^{1/q},\quad q\in (0, \infty], \, p\in (0,\infty], 
\end{align*}    
\footnotetext[1]{{We omit the limiting case $p=\infty$ which requires a different definition.}}where in the case $q=\infty$, the $\ell^q(\mathbb Z)$ norms  above must be replaced by the $\ell^\infty(\mathbb Z)$ norms.
 The homogeneous  Triebel-Lizorkin space $\dot{F}^{s}_{p,q}(\R^n)$  $($resp. Besov space $\dot{B}^{s}_{p,q}(\R^n)$$)$  is defined as the space of tempered distributions  modulo polynomials $f\in \Qs{}{}$  such that $\|f\|_{\dot{F}^s_{p,q}}<+\infty$ $($resp.  $\|f\|_{\dot{B}^s_{p,q}}<+\infty$$)$.
\end{definition}
It is easy to see from the definition that 
\begin{equation}\label{eq:Trie-Besov-emb}
\dot{B}^{s}_{p,\min(p,q)}(\R^n) \hookrightarrow \dot{F}^{s}_{p,q}(\R^n)  \hookrightarrow \dot{B}^{s}_{p,\max(p,q)}(\R^n).   
\end{equation}

More generally, we introduce the following definition.
\begin{definition}
Let $s\in \R$ and $ q \in (0,\infty]$. For any quasi-Banach function space $X$ on  $\R^n$ with quasi-normed $\|\cdot\|_{X}$, we  define
\begin{align*}
 \|f\|_{\dot{F}^s_{q}X(\R^n)}:= \big\|\big\{\sum_{j\in \mathbb{Z}}\big(2^{js}\Dyadic f\big)^q\big \}^{1/q}\big\|_{X},\quad q\in (0, \infty],  \\
 \|f\|_{\dot{B}^s_{q}X(\R^n)}:= \big\{\sum_{j\in \mathbb{Z}}\big (2^{js}\|\Dyadic f\|_{X}\big)^q\big\}^{1/q},\quad q\in (0, \infty]. 
\end{align*}    
 With this, we set
$$\dot{F}^s_{q}X(\R^n):=\{f\in \Qs{}{}:   \|f\|_{\dot{F}^s_{q}X(\R^n)}<+\infty\}$$
and 
$$\dot{B}^s_{q}X(\R^n):=\{f\in \Qs{}{}:   \|f\|_{\dot{B}^s_{q}X(\R^n)}<+\infty\}.$$
\end{definition}

In applications, one of our choices will be the weighted Lebesgue space $X=L^s(w)$, $s\in (0,\infty)$, where $w$ is an $A_\infty$ weight in $\R^n$. Recall that 
a nonnegative and locally integrable function $w$ is an  $A_\infty$ weight if $w$ is an $A_p$ weight for some $p > 1$, which means that 
$$[w]_{A_p}:= \left(\sup_{{\rm balls} \, B\subset\R^n} \fint_{B} w dy\right) \left(\sup_{B} \fint_{B} w^{\frac{1}{p-1}} dy\right)^{p-1} <+\infty $$
for some $p>1$. 
Moreover, it is well-known that the $A_p$ class, for any $p>1$, contains the $A_1$ class  which consists of nonnegative functions $w\in L^1_{\rm loc}(\R^n)$ such that 
$$ \sup_{r>0}\frac{1}{|B_r(x)|}\int_{B_r(x)} w(y)dy \leq C \, w(x)$$ 
for a.e. $x\in\R^n$. The best constant $C$ in the above inequality, denoted by $[w]_{A_1}$, is called the $A_1$ constant of $w$.

Another example is the Sobolev multiplier type space $X=\Ms{\alpha}{r}{p}(\R^n)$ that we now define.
\begin{definition}
 Let $p > 0$, $\alpha >0$,  and $r>1$ with $\alpha r <n$. The Sobolev multiplier type space $\Ms{\alpha}{r}{p}(\R^n)$ is 
the space of functions $f\in L^p_{\rm loc}(\mathbb{R}^n)$
such that the trace inequality 
\begin{align}\label{eq:trace-ineq}
	\left(\int_{\mathbb{R}^n} {\bf I}_\alpha[h]^r |f|^p dx\right)^{\frac{1}{p}} \leq C \left(\int_{\R^n} h^r dx\right)^{\frac1p} 
\end{align} 
holds for all $h\in L^r(\mathbb{R}^n), h\geq 0$. We equip $\Ms{\alpha}{r}{p}(\R^n)$ with the quasi-norm defined as the least possible constant $C$ in \eqref{eq:trace-ineq}.   
\end{definition}

In practice, it is not easy to work with \eqref{eq:trace-ineq}. The following characterizations  for the quasi-norm of $\Ms{\alpha}{r}{p}(\R^n)$ will be useful.
\begin{proposition}\label{W-char}
Let $p > 0$, $\alpha >0$,  and $r>1$ with $\alpha r <n$. Then have have 

\begin{equation}\label{TraceMAD} 
\|f\|_{\Ms{\alpha}{r}{p}(\R^n)} \simeq \sup_{K}\left(\frac{\int_{K}|f(x)|^{p}dx}{\Rcap{\alpha}{r}(K)}\right)^{1/p}, 
\end{equation}
where the supremum is taken over all compact sets $K\subset\R^n$;  and also,
\begin{equation}\label{OPW}
\|f\|_{\Ms{\alpha}{r}{p}(\R^n)} \simeq \sup_{w} \left(\int_{\R^n} |f|^{p} w dx\right)^{\frac1p}, 
\end{equation}
where the supremum is taken over  all $A_1$ weights $w$ in $\R^n$ such that  $w$ is quasi-continuous with respect to ${\rm Cap}_{{\bf I}_{\alpha}, r}$,  
$$\int_{0}^\infty {\rm Cap}_{{\bf I}_\alpha, r} (\{ w>t\}) dt \leq 1, $$
and $[w]_{A_1}\leq \overline{{\bf c}}(n,\alpha, r)$ for some constant $\overline{{\bf c}}(n,\alpha, r)>0$.
\end{proposition}

Recall  that a function $v:\R^n\rightarrow [-\infty, \infty]$ is said to be  quasi-continuous with respect to ${\rm Cap}_{{\bf I}_{\alpha}, r}$
if  for any  $\epsilon>0$ there exists an open set $O$ such that ${\rm Cap}_{{\bf I}_{\alpha}, r}(O)<\epsilon$
and the restriction of $v$ to  $\R^n\setminus O$  is a continuous function.

For the equivalence in \eqref{TraceMAD}, see  \cite{MS2, AH}. On the other hand, the equivalence \eqref{OPW} was obtained in \cite[Theorem 9.2]{OP1}. 
The constant $\overline{{\bf c}}$ can also be taken to be $\overline{{\bf c}}=[|\cdot|^{-n+\alpha}]_{A_1}$, which is independent of $r$ (see \cite[Theorem 6.7]{OP2}). Both \cite{OP1} and \cite{OP2} considered only the case $p>1$, but the proofs there also apply to all exponents $p>0$.

\begin{remark} Let $p > 0$, $\alpha >0$,  and $r>1$ with $\alpha r <n$.  It follows from the Sobolev embedding theorem that 
\begin{equation*}
|K|^{1-\alpha r/n} \lesssim \Rcap{\alpha}{r}(K),
\end{equation*}  
for any measurable set $K\subset\R^n$, and thus we have the continuous embedding 
\begin{equation}\label{WeakEmb}
L^{\frac{n p}{\alpha r},\infty}(\R^n)\hookrightarrow \Ms{\alpha}{r}{p}(\R^n),
\end{equation}
where $L^{s, \infty}(\R^n), s>0$, denotes the weak $L^s(\R^n)$ space.
More generally, for any $0<\epsilon< \frac{n}{\alpha r}-1,$
\begin{equation}\label{MorreyEmb}
\mathcal{M}^{p(1+\epsilon), \alpha r(1+\epsilon)}(\R^n)\hookrightarrow \Ms{\alpha}{r}{p}(\R^n) \hookrightarrow \mathcal{M}^{p, \alpha r}(\R^n),
\end{equation}
where $\mathcal{M}^{p, \lambda}(\R^n)$, $p>0, \lambda\in (0,n)$, denotes the Morrey space defined by
$$\mathcal{M}^{p, \lambda}(\R^n)=\left\{f\in L^p_{\rm loc}(\R^n): \left\|f\right\|_{\mathcal{M}^{p, \lambda}(\R^n)}<+\infty \right\},$$
with
$$\left\|f\right\|_{\mathcal{M}^{p, \lambda}(\R^n)}:= \sup_{B_R} \left(R^{\lambda -n}\int_{B_R} |f|^p dy\right)^{\frac1p}.$$
The second embedding in \eqref{MorreyEmb} follows from the fact that $\Rcap{\alpha}{r}(B_R)\simeq R^{n-\alpha r}$, whereas the first embedding is an analogue of a result due to Fefferman and Phong; see \textnormal{\cite{Fef}} and \textnormal{\cite{SW}}.
\end{remark}

By  \eqref{eq:Trie-Besov-emb}, \eqref{WeakEmb}, and \eqref{MorreyEmb}, we can easily deduce that, for  $0<\epsilon<\frac{n}{\alpha r} -1,$
\begin{equation*}
\dot{B}^{s}_{q,\min(\frac{np}{\alpha r},q)}(\R^n)  \hookrightarrow  \dot{F}^s_{q}L^{\frac{np}{\alpha r}, \infty}(\R^n) \hookrightarrow  \dot{F}_{q}^{s}\mathcal{M}^{p(1+\epsilon), \alpha r(1+\epsilon)}(\R^n), 
\end{equation*}
and 
\begin{equation*}
 \dot{F}_{q}^{s}\mathcal{M}^{p(1+\epsilon), \alpha r(1+\epsilon)}(\R^n) \hookrightarrow  \TrMs{s}{\alpha}{r}{p}{q}(\R^n)
 \hookrightarrow \dot{F}_{q}^{s} \mathcal{M}^{p, \alpha r}(\R^n)
\end{equation*}
for any $s\in \R$ and $0<q\leq \infty$.

On the other hand,    $\TrMs{s}{\alpha}{r}{p}{q}(\R^n)$  is  also included in some endpoint  Besov space, which is convenient for different purposes.
\begin{lemma}\label{lem:embedding-endpoint-besov}
Let  $(s,\alpha)\in \R\times \R_+$ and $(p,q,r)\in (0,\infty)\times (0,\infty]\times (1,\infty)$ be such that $\alpha r< n$.   
 Then we have the continuous embedding 
\begin{equation*}
  \TrMs{s}{\alpha}{r}{p}{q} (\R^n)\hookrightarrow \dot{B}^{s-\frac{\alpha r}{p}}_{\infty,\infty}(\R^n).  
\end{equation*}
\end{lemma}
\begin{proof} 
Since 
$$\TrMs{s}{\alpha}{r}{p}{q} (\R^n)\hookrightarrow \dot{F}_{\infty}^{s} \mathcal{M}^{p, \alpha r}(\R^n) \hookrightarrow \dot{B}_{\infty}^{s} \mathcal{M}^{p, \alpha r}(\R^n),$$
it suffices to show that 
\begin{equation}\label{MtoB}
\dot{B}_{\infty}^{s} \mathcal{M}^{p, \alpha r}(\R^n) \hookrightarrow \dot{B}^{s-\frac{\alpha r}{p}}_{\infty,\infty}(\R^n).
\end{equation}


Let $f\in \dot{B}_{\infty}^{s} \mathcal{M}^{p, \alpha r}(\R^n)$ and,  by H\"older's inequality, we may assume  that $0<p<1$.
 For each $j\in \mathbb Z$, let $\psi_j \in C^\infty_c(\R^n)$, $0\leq \psi_j\leq 1$, be such that $\psi_j=1$ for $ |\xi|\leq 2^{j+1}$ and $\psi_j=0$ for $ |\xi|\geq 2^{j+2}$. With $\varphi_j$ as in \eqref{varphij},  as   ${\rm supp} (\varphi_j)\subset B_{2^{j+1}}(0)$, we have $\varphi_j=\psi_j \varphi_j$ and hence
\begin{equation}\label{repvar}
  \varphi_j(D)f(x) = \mathcal{F}^{-1}(\psi_j \varphi_j\mathcal{F}f)(x)=K_j\ast (\varphi_j(D)f)(x),
  \end{equation}
where  $K_j=\mathcal{F}^{-1}(\psi_j).$ Note that 
\begin{equation}\label{Kinfty}
\|K_j\|_{L^{\infty}} \lesssim \, 2^{j n}.
\end{equation}

Now it follows from  \eqref{repvar} that 
\begin{align}\label{I+II}
\big|\varphi_j(D)f(x)\big|&\leq \int_{B_{2^{-j}}(x)}|K_j(x-y)|| \varphi_j(D)f(y)|dy \\
&\quad +\sum_{i=j-1}^{\infty} \int_{B_{2^{-i}}(x)\setminus B_{2^{-i-1}}(x)}|K_j(x-y)|| \varphi_j(D)f(y)|dy \nonumber\\
&=:I+II. \nonumber
 \end{align}

On the other hand, by \eqref{Kinfty}, we have
\begin{align*}
I&\lesssim  2^{-jsp}\|K_j\|_{L^{\infty}}  \|\varphi_j(D)f\|_{L^\infty}^{1-p}   \int_{B_{2^{-j}}(x)}\big(2^{js}|\varphi_j(D)f(y)|\big)^pdy  \\
&\lesssim  2^{-jsp}2^{jn}  \|\varphi_j(D)f\|_{L^\infty}^{1-p}  \|f\|^p_{\dot{B}_{\infty}^{s} \mathcal{M}^{p, \alpha r}(\R^n)}  2^{-j(n-\alpha r)} \\
&\lesssim 2^{-(sp-\alpha r)j} \|\varphi_j(D)f\|_{L^\infty}^{1-p}    \|f\|^p_{\dot{B}_{\infty}^{s} \mathcal{M}^{p, \alpha r}(\R^n)}.
\end{align*}

Similarly, for each integer  $i\geq j-1$, we estimate
\begin{align*}
& \int_{B_{2^{-i}}(x)\setminus B_{2^{-i-1}}(x)}|K_j(x-y)|| \varphi_j(D)f(y)|dy\\
&\lesssim 2^{-jsp}\|K_j\|_{L^{\infty}}  \|\varphi_j(D)f\|_{L^\infty}^{1-p}   \int_{B_{2^{-i}}(x)}\big(2^{js}|\varphi_j(D)f(y)|\big)^pdy  \\
&\lesssim 2^{-jsp}2^{jn}  \|\varphi_j(D)f\|_{L^\infty}^{1-p}  \|f\|^p_{\dot{B}_{\infty}^{s} \mathcal{M}^{p, \alpha r}(\R^n)}  2^{-i(n-\alpha r)}.
\end{align*}
 
 Thus, it follows that 
\begin{align*}
II&\lesssim 2^{-sjp}2^{j n} \|\varphi_j(D)f\|_{L^\infty}^{1-p}  \|f\|^p_{\dot{B}_{\infty}^{s} \mathcal{M}^{p, \alpha r}(\R^n)}  \sum_{i=j-1}^\infty 2^{-i(n-\alpha r)}  \\
&\lesssim 2^{-(s p- \alpha r)j}  \|\varphi_j(D)f\|_{L^\infty}^{1-p} \|f\|^p_{\dot{B}_{\infty}^{s} \mathcal{M}^{p, \alpha r}(\R^n)}.
\end{align*}

Assume for the monent that  $\|\varphi_j(D)f\|_{L^\infty}<+\infty$. Then   by \eqref{I+II} and the estimates for $I$ and $II$, we obtain
$$\|\varphi_j(D)f\|_{L^\infty} \lesssim 2^{-(s - \alpha r/p)j}   \|f\|_{\dot{B}_{\infty}^{s} \mathcal{M}^{p, \alpha r}(\R^n)}$$ 
for all $j\in \mathbb Z$, which yields \eqref{MtoB} as desired.

To remove the assumption  $\|\varphi_j(D)f\|_{L^\infty}<+\infty$, let $\rho \in C_c^\infty(B_1(0))$ be such that $\int_{\R^n} \rho dx=1$. Set $h(x)=\mathcal{F}^{-1}(\rho)(x)$ and note that $h(0)=1$. As $\varphi_j(D)f$ is of polynomial growth, we see that 
$ h(t\cdot) \varphi_j(D)f \in \mathcal{S}(\R^n)\subset L^\infty(\R^n)$ for any $t>0$. Moreover,
$${\rm supp}(\mathcal{F}\left(h(t\cdot) \varphi_j(D)f\right))\subset B_{t+2^{j+1}}(0).$$
Thus, by modifying the argument above, we find, for any $t>0$ and $j\in \mathbb Z$,
$$\| h(t\cdot) \varphi_j(D)f\|_{L^\infty} \lesssim (t+ 2^{j})^{n/p}  2^{-j n/p} 2^{-(s - \alpha r/p)j}   \|f\|_{\dot{B}_{\infty}^{s} \mathcal{M}^{p, \alpha r}(\R^n)}.$$ 
Then letting $t\rightarrow 0$, we arrive at \eqref{MtoB}.
\end{proof}

We now recall the notion vanishing at infinity for tempered distributions in $\R^n$  (see \cite[Definition 4.3]{Bou}).
\begin{definition} We say that a distribution $f\in \mathcal{S}'(\R^n)$ vanishes at infinity if 
$$\lim_{\lambda \rightarrow 0}\langle f, \lambda^n \varphi(\lambda \cdot)\rangle=0, \qquad \forall \,\varphi\in \mathcal{S}(\R^n).$$ 
The set of such distributions is denoted by $\mathcal{S}'_0(\R^n)$.
\end{definition}

By definition, $\dot{B}^\beta_{\infty, \infty}(\R^n)$, $\beta\in \R$, is a subspace of $\mathcal{S}'/\mathcal{P}(\R^n)$. However, in the case $\beta<0$, it can be regarded as a subspace of $\mathcal{S}'$. 
More precisely, we mention here the following result, which is a realization of $\dot{B}^\beta_{\infty, \infty}(\R^n)$, when $\beta< 0$, in $\mathcal{S}'$ (see \cite[Proposition 4.6]{Bou}).

\begin{proposition}\label{BoudReal}
Let $\beta<0$.  For any $f\in \dot{B}^\beta_{\infty, \infty}(\R^n)$,  the series  $\sum_{j\in \mathbb{Z}} \Dyadic f$ converges in $\mathcal{S}'(\R^n)$. The map
\begin{align*}
\Phi: \dot{B}^\beta_{\infty, \infty}(\R^n) \rightarrow   \mathcal{S}'(\R^n) {\rm ~given~by~} \Phi([f]):= \sum_{j\in \mathbb{Z}} \Dyadic f
\end{align*}
is a realization of $\dot{B}^\beta_{\infty, \infty}(\R^n)$ in $\mathcal{S}'$ in the sense that $\Phi$ is linear, continuous, and 
$\Phi([f])\in [f]$, the equivalent class of $f$, for all $[f]\in \dot{B}^\beta_{\infty, \infty}(\R^n)$. Moreover, the image $\Phi(\dot{B}^\beta_{\infty, \infty}(\R^n))$ is characterized as the set of all $g\in \mathcal{S}'$ such that $[g]\in \dot{B}^\beta_{\infty, \infty}(\R^n)$ and  $g$ vanishes  at infinity.
\end{proposition}

We remark that for $g\in \Phi(\dot{B}^\beta_{\infty, \infty}(\R^n))$, one has $g=\sum_{j\in \mathbb{Z}} \Dyadic f$ in $\mathcal{S}'(\R^n)$ for some $f\in \mathcal{S}'(\R^n)$. But since 
$\sum_{j\in \mathbb{Z}} \Dyadic f=f+P$ for a polynomial $P$, it follows that we have the representation $g=\sum_{j\in \mathbb{Z}} \Dyadic g$ in $\mathcal{S}'(\R^n)$.

Since $\TrMs{s}{\alpha}{r}{p}{q}(\R^n)\hookrightarrow \dot{B}^{s-\frac{\alpha r}{p}}_{\infty, \infty}(\R^n)$ by Lemma \ref{lem:embedding-endpoint-besov}, we obtain the following realization of 
$\TrMs{s}{\alpha}{r}{p}{q}(\R^n)$ in $\mathcal{S}'(\R^n)$.
\begin{corollary}  Let  $(s,\alpha)\in \R\times \R_+$ and $(p,q,r)\in (0,\infty)\times (0,\infty]\times (1,\infty)$ be such that $\alpha r< n$ and $s<\frac{\alpha r}{p}$.   Then the map 
\begin{align*}
\Phi: \TrMs{s}{\alpha}{r}{p}{q}(\R^n) \rightarrow   \mathcal{S}'(\R^n) {\rm ~given~by~} \Phi([f]):= \sum_{j\in \mathbb{Z}} \Dyadic f
\end{align*}
is a realization of $\TrMs{s}{\alpha}{r}{p}{q}(\R^n)$ in $\mathcal{S}'$. Moreover,  the image $\Phi(\TrMs{s}{\alpha}{r}{p}{q}(\R^n))$ is characterized as the set of all $g\in \mathcal{S}'$ such that $[g]\in \TrMs{s}{\alpha}{r}{p}{q}(\R^n)$ and  $g$ vanishes  at infinity. 
\end{corollary}

\begin{remark} Let $\beta<0$ and suppose that $E$ is a subspace of  $\dot{B}^{\beta}_{\infty, \infty}(\R^n)$. By Proposition  \ref{BoudReal},
in what follows and when appropriate, we will identify $E$ with $\Phi(E)$ and treats $E$ as a subspace 
of $\mathcal{S}'_0(\R^n)$.  For example, we identify $\dot{B}^{\beta}_{\infty, \infty}(\R^n)$ with $\Phi(\dot{B}^{\beta}_{\infty, \infty}(\R^n))$  and $\TrMs{s}{\alpha}{r}{p}{q}(\R^n)$ with 
$\Phi(\TrMs{s}{\alpha}{r}{p}{q}(\R^n))$ when $s<\frac{\alpha r}{p}$.
\end{remark}

We will need the following characterizations of weighted homogeneous Triebel-Lizorkin spaces with negative order when the weights involved  are in the $A_\infty$ class.

\begin{theorem}\label{Bui}  Let $s<0$, $(p, q)\in (0,\infty)\times (0, \infty]$, and  $w\in A_\infty$. Assume that $\nu$ is a function in $\mathcal{S}(\R^n)$ that satisfies the Tauberian condition, i.e.,   
\begin{equation}\label{Tau}
\forall \,\xi \not= 0 \quad {\rm there ~exists } \quad  t>0 \quad {\rm such~that} \quad \mathcal{F}(\nu)(t\xi)\not = 0.
\end{equation}
Then we have 
\begin{equation}\label{UB-97}
\left\| f \right\|_{\dot{F}_q^s L^p(w)} \leq C \left\| \left( \int_0^\infty | t^{-s} \langle f, \nu_t(x-\cdot) \rangle|^q \frac{dt}{t} \right)^{\frac1q}\right\|_{L^p(w)},
\end{equation}
where $\nu_t(\cdot)= t^{-n} \nu(\cdot/t)$, and there exists a polynomial $P$ such that 
\begin{equation}\label{LB-96}
 \left\| \left( \int_0^\infty | t^{-s} \langle f -P, \nu_t(x-\cdot) \rangle|^q \frac{dt}{t} \right)^{\frac1q}\right\|_{L^p(w)} \leq C\, \left\| f \right\|_{\dot{F}_q^s L^p(w)}.
\end{equation}
 With an appropriate modification, inequalities \eqref{UB-97} and \eqref{LB-96} also hold for $q=\infty$. 
\end{theorem}

Inequalities  \eqref{UB-97} and  \eqref{LB-96} were obtained in \cite[Theorem 2]{B-97} and \cite[Theorem 3.1]{B-96}, respectively (see also \cite{B-84} and \cite{B-94}).
We notice  from the proof of \cite[Theorem 3.1]{B-96}  that 
 one can take $P=0$ when $f$ has the representation 
 \begin{equation}\label{Canon}
f=\sum_{j\in\mathbb{Z}} \Dyadic f \quad {\rm in~} \mathcal{S}'(\R^n).
\end{equation}

We now use   Theorem \ref{Bui} and Proposition \ref{W-char}  to obtain  useful characterizations of $\TrMs{s}{\alpha}{r}{p}{q}(\R^n)$, $s<0$, by smooth cut-off functions or by temperatures.

\begin{theorem}\label{temChar}
Let $s<0$, $p\in (0,\infty)$, $q\in (0,\infty]$, $r\in (1,\infty)$ and $\alpha>0$ such that $\alpha r<n$. Let $\nu\in C_c^\infty(\R^n)$ be such that  $\upchi_{B_{1/2}(0)} \leq \nu\leq \upchi_{B_1(0)}$, and  set 
$\nu_t(\cdot)= t^{-n} \nu(\cdot/t), t>0$. 
Then  we have 
\begin{equation}\label{UB-nu}
\left\| f \right\|_{\TrMs{s}{\alpha}{r}{p}{q}(\R^n)} \lesssim \left\| \left( \int_0^\infty | t^{-s} \langle f, \nu_t(x-\cdot) \rangle|^q \frac{dt}{t} \right)^{\frac1q}\right\|_{\Ms{\alpha}{r}{p}(\R^n)},
\end{equation}
\begin{equation}\label{UP-M}
\left\| f \right\|_{\TrMs{s}{\alpha}{r}{p}{q}(\R^n)} \lesssim \left\| \left( \int_0^\infty | t^{-\frac{s}{2}} e^{t\Delta} f(x)|^q \frac{dt}{t} \right)^{\frac1q}\right\|_{\Ms{\alpha}{r}{p}(\R^n)},
\end{equation}
and there exists a polynomial $P$ such that 
\begin{equation}\label{LB-nu}
 \left\| \left( \int_0^\infty | t^{-s} \langle f - P, \nu_t(x-\cdot) \rangle|^q \frac{dt}{t} \right)^{\frac1q}\right\|_{\Ms{\alpha}{r}{p}(\R^n)}\lesssim \left\| f \right\|_{\TrMs{s}{\alpha}{r}{p}{q}(\R^n)},
 \end{equation}
\begin{equation}\label{LB-M}
 \left\| \left( \int_0^\infty | t^{-\frac{s}{2}} e^{t\Delta}(f-P) |^q \frac{dt}{t} \right)^{\frac1q}\right\|_{\Ms{\alpha}{r}{p}(\R^n)} \lesssim\, \left\| f \right\|_{\TrMs{s}{\alpha}{r}{p}{q}(\R^n)}.
\end{equation}
 With an appropriate modification, inequalities \eqref{UB-nu}--\eqref{LB-M} also hold for $q=\infty$. Moreover, we can take $P=0$ in \eqref{LB-nu} and \eqref{LB-M} when $f$ has the representation \eqref{Canon}.
\end{theorem}
\begin{proof} Inequalities \eqref{UB-nu} and \eqref{LB-nu} follow from  Proposition \ref{W-char} and Theorem \ref{Bui} with $\nu\in C_c^\infty(\R^n)$  such that  $\upchi_{B_{1/2}(0)} \leq \nu\leq \upchi_{B_1(0)}$. Note that the Tauberian condition \eqref{Tau} holds in this case.
Similarly, inequalities \eqref{UP-M} and \eqref{LB-M} follow from   Proposition \ref{W-char} and Theorem  \ref{Bui}, with $\nu=(4\pi)^{-n/2} e^{-|\cdot|^2/4}$, combined with an appropriate change of variables. 
\end{proof}

\section{Semilinear parabolic equations with nonnegative measures as initial data}

We now recall  the following fundamental result on  the existence of nonnegative global in time solutions to the Cauchy problem
\begin{align}
 \begin{cases}\label{ABC}
 \partial_tu-\Delta u=|u|^{b-1}u \quad {\rm ~in~} \R^n\times (0,\infty),\\
 u(\cdot,0)= u_0 \quad {\rm ~in~} \R^n,
 \end{cases}   
\end{align}
 with  nonnegative  measure initial data  that was obtained by Baras and Pierre  in  \cite[Theorem 2.1]{BP} (see also \cite[Theorem 2.12]{KV}).
\begin{theorem}\label{BP} Let $b>1$. Suppose that $u_0$ is nonnegative measure on $\R^n$. Then equation 
\eqref{ABC} has a nonnegative solution $u\in L^b_{\rm loc}(\R^n\times (0,\infty))$ if and only if 
\begin{equation*}
\int_{0}^\infty\int_{\R^n} e^{t\Delta} u_0(x) h(x,t) dxdt\leq  \frac{b-1}{b^{b'}} \int_{0}^\infty\int_{\R^n} \frac{h(x,t)^{b'}}{ [\widecheck{\Heat}*h(x,t)]^{b'-1}} dxdt 
\end{equation*}
holds for all nonnegative $h\in L^\infty_{c}(\R^n\times (0,\infty))$ (the space of bounded and compactly supported functions on $\R^n\times (0,\infty)$). Here, 
$$\widecheck{\Heat}(x,t)=\Heat(x,-t)=  \frac{\upchi_{(-\infty,0)}(t)}{(-4\pi t)^{n/2}} e^{\frac{|x|^2}{4t}},\qquad (x,t)\in\R^{n+1},$$
is the adjoint of the heat kernel $\Heat(x,t)$. 
\end{theorem}

We will need the following lemma but
 only in the case $p>1$.
\begin{lemma}\label{Fchar} Let $\alpha, s \in (0,n)$, $p> 0$ and $1<r<\frac{n}{\alpha}$. Suppose that $\mu$ is a nonnegative locally finite measure in $\R^n$.  Then we have 
\begin{equation}\label{FandR}
 c_1 \|{\bf I}_s[\mu]\|_{\Ms{\alpha}{r}{p}(\R^n)} \leq \|\mu\|_{\TrMs{-s}{\alpha}{r}{p}{q}(\R^n)} \leq c_2 \|{\bf I}_s[\mu]\|_{\Ms{\alpha}{r}{p}(\R^n)}
\end{equation}
for all $0<q\leq \infty$, where in the first inequality we assume that $\mu$ has the representation
$$\mu=\sum_{j\in\mathbb{Z}} \Dyadic \mu \quad {\rm in~} \mathcal{S}'(\R^n).$$
\end{lemma}

\begin{proof}
We first recall the weighted bound
\begin{equation}\label{MWB}
\left\|\left(\int_{0}^\infty \Big(\frac{\mu(B_t(\cdot))}{t^{n-s}}\Big)^\rho \frac{dt}{t} \right)^{\frac1\rho} \right\|_{L^p(w)}  \lesssim \left\| \sup_{t>0} \frac{\mu(B_t( \cdot))}{t^{n-s}}  \right\|_{L^p(w)},
\end{equation}
which holds for all $0<p, \rho<\infty$ and weights $w\in A_\infty$. The case $\rho\geq 1$ follows from the result of Muckenhoupt and Wheeden \cite{MWh} and the case $0<\rho<1$ can be found in \cite{JPW} and \cite{HJ}.

 Let $\nu\in C_c^\infty(\R^n)$ be such that  $\upchi_{B_{1/2}(0)} \leq \nu\leq \upchi_{B_1(0)}$, and set $\nu_t(\cdot)= t^{-n} \nu(\cdot/t)$.
To   prove the first inequality in \eqref{FandR}, it is enough to consider the  case $q=\infty$.  By Theorem \ref{Bui} with   $f=\mu \geq 0$  and the above choice of $\nu$,   we find
\begin{align}\label{Lweight}
\| \mu \|_{\dot{F}_\infty^{-s} L^p(w)} \geq  C_1 \left\| \sup_{t>0} \frac{\mu(B_t(\cdot))}{t^{n-s}}  \right\|_{L^p(w)} \geq C_2 \left\| {\bf I}_s[\mu]  \right\|_{L^p(w)},
\end{align}
where $w\in A_\infty$ and we used  \eqref{MWB} with $\rho=1$ in the last inequality. As $A_1\subset A_\infty$, by combining Proposition \ref{W-char} with \eqref{Lweight} we obtain the first bound in \eqref{FandR}. 

Conversely, in order to prove the second inequality in \eqref{FandR}, we may assume that $q\not=\infty$. For $w\in A_\infty$, by  Theorem \ref{Bui} with $f=\mu \geq 0$, we have 
\begin{equation}
\label{Uweight}
\begin{aligned}
\| \mu \|_{\dot{F}_q^{-s}L^p(w)} &\lesssim\left\|\left(\int_{0}^\infty \Big(\frac{\mu(B_t(\cdot))}{t^{n-s}}\Big)^q \frac{dt}{t} \right)^{\frac1q} \right\|_{L^p(w)}\\
&\lesssim \left\| \sup_{t>0} \frac{\mu(B_t( \cdot))}{t^{n-s}}  \right\|_{L^p(w)} \\
& \lesssim \left\| {\bf I}_s[\mu]  \right\|_{L^p(w)},
\end{aligned}
\end{equation}
where we used \eqref{MWB} with $\rho=q$ in the penultimate inequality. Finally, combining \eqref{Uweight}
and Proposition \ref{W-char}, we obtain the second bound in \eqref{FandR} as desired.
\end{proof}

Our main result in this section is the following.

\begin{theorem}\label{BP-Hung} Let $b>1+ \frac2n$.  Suppose that $u_0$ is nonnegative locally finite measure on $\R^n$.  The following statements are equivalent.

\rm{(i)} The equation
\begin{align}\label{eq:HE-kappa}
 \begin{cases}
 \partial_tu-\Delta u=|u|^{b-1}u \quad {\rm ~in~} \R^n\times (0,\infty),\\
 u(\cdot,0)=\kappa \, u_0 \quad {\rm ~in~} \R^n,
 \end{cases}   
\end{align}
has a nonnegative solution $u\in L^b_{\rm loc}(\R^n\times (0,\infty))$ for some $\kappa>0$.

\rm{(ii)} There exists a constant $c_1>0$ such that 
\begin{equation*}
u_0(K) \leq  c_1\, {\rm Cap}_{{\bf I}_{\frac{2}{b}}, b'}(K)
\end{equation*}
holds for all  compact sets $K\subset\R^n$.

\rm{(iii)} There exists a constant $c_2>0$ such that 
\begin{equation*}
(u_0\otimes\delta_0)(E) \leq  c_2\, {\rm cap}_{2, 1, b'}(E)
\end{equation*}
holds  for all compact sets $E\subset\R^{n+1}$, where $\delta_0$ is the Dirac measure supported at $0$ in $\R$.  
 
\rm{(iv)} There exists a constant $c_3>0$ such that 
\begin{equation*}
\int_{E}\mathbb{I}_2(u_0\otimes\delta_0)^b dxdt \leq  c_3\, {\rm cap}_{2, 1, b'}(E)
\end{equation*}
holds  for all compact sets $E\subset\R^{n+1}$.

\rm{(v)} There exists a constant $c_4>0$ such that 
\begin{equation*}
\int_{E} \left[e^{t\Delta} u_0(x)\right]^b   dx dt \leq  c_4\, {\rm cap}_{2, 1, b'}(E)
\end{equation*}
holds  for all compact sets $E\subset\R^{n+1}$.

{\rm (vi)} There exists a constant $c_5>0$ such that 
\begin{equation}\label{I2b}
\int_{K} \left({\bf I}_{\frac{2}{b}}[u_0]\right)^b dx \leq  c_5\, {\rm Cap}_{{\bf I}_{\frac{2}{b}}, b'}(K)
\end{equation}
holds  for all compact sets $K\subset\R^{n}$.

{\rm (vii)} $u_0$ admits the representation 
\begin{equation}\label{pre-u0}
u_0=\sum_{j\in\mathbb{Z}} \Dyadic u_0 \quad {\rm in~} \mathcal{S}'(\R^n)
\end{equation}
and $u_0\in \TrMs{-\frac{2}{b}}{\frac{2}{b}}{b'}{b}{q}(\R^n)$ for any $0<q\leq \infty$. 
\end{theorem}

\begin{proof}  The equivalence  ${\rm (ii)}\Leftrightarrow {\rm (iii)}$ follows from Theorem \ref{thm:cap-equiv} and the fact that 
\begin{equation}\label{0-cap}
{\rm cap}_{\Heat, b'}(K\times\{0\})\ \simeq {\rm Cap}_{\mathbf{I}_{\frac{2}{b}}, b'}(K)
\end{equation}
for all compact sets $K\subset\R^n$ (see \cite[Corollary 2.20]{Q-H}). That ${\rm (iii)} \Leftrightarrow {\rm (iv)}$ follows from Theorem \ref{thm:cap-equiv} and \cite[Theorem 2.36]{Q-H}. 
The implication ${\rm (iv)}\Rightarrow {\rm (v)}$ is obvious and by Theorem \ref{thm:cap-equiv} the implication 
${\rm (v)}\Rightarrow {\rm (i)}$ is a consequence of  \cite[Corollary 2.42]{Q-H}.  Moreover, the equivalence  $\rm{(ii)}\Leftrightarrow (vi)$ was originally obtained in 
\cite[Theorem 2.1]{MV}.
 Thus, it remains to show ${\rm (i)}\Rightarrow {\rm (ii)}$ and ${\rm (vi)} \Leftrightarrow {\rm (vii)}$.

$\bullet$ {\bf Proof of  ${\rm (i)}\Rightarrow {\rm (ii)}$:} Suppose that \eqref{eq:HE-kappa}, with some $\kappa>0$,  has a nonnegative solution $u\in L^q_{\rm loc}(\R^n\times(0,\infty))$.
Then by Theorem \ref{BP} we have 
\begin{equation}\label{BP-cond}
\int_{0}^\infty\int_{\R^n} e^{t\Delta} u_0(x) h(x,t) dxdt\leq  \frac{b-1}{\kappa \, b^{b'}} \int_{0}^\infty\int_{\R^n} \frac{h(x,t)^{b'}}{ [\widecheck{\Heat}*h(x,t)]^{b'-1}} dxdt 
\end{equation}
 for all nonnegative $h\in L^\infty_{c}(\R^n\times (0,\infty))$. 
Observe then that \eqref{BP-cond} also holds for all 
$h\in L^{b'}_{c}(\R^n\times [0,\infty)), h\geq 0,$ by Lebesgue Dominated Convergence Theorem and Fatou's lemma.
For $h\in L^{b'}_{c}(\R^n\times [0,\infty)), h\geq 0$, let $$v(x,t)=\widecheck{\Heat}*h(x,t). $$ 
As $-\partial_t v -\Delta v=h$, by Theorem \ref{BP} and Fubini's theorem, we have 
\begin{align}\label{abstest}
 \int_{\R^n} v(\cdot,0) d u_0 & = \int_{0}^\infty\int_{\R^n} e^{t\Delta}u_0(x) h(x,t) dxdt \\
  & \leq \frac{b-1}{\kappa b^{b'}} \int_{0}^\infty\int_{\R^n} \frac{(-\partial_t v -\Delta v)^{b'}}{v^{b'-1}} dx dt. \nonumber
\end{align}
We now claim  that we can deduce from \eqref{abstest} the following bound
\begin{align}\label{trace-qprime}
 \int_{\R^n}  \Phi(\cdot, 0)^{2b'} d u_0  \leq c(n, b, \kappa)  \int_{0}^\infty\int_{\R^n} |\partial_t \Phi|^{b'} + |D^2 \Phi|^{b'} dx dt
\end{align}
for all $\Phi\in C_c^\infty(\R^{n+1})$ and $0\leq \Phi\leq 1$.   

To this end, let $\varphi_k=\varphi(x/k,t/k), k=1, 2, \dots,$ where $\varphi$ is a cut-off function in $C_c^\infty(\mathbb{B}_2(0))$, $0\leq \phi \leq 1$, and $\varphi=1$ on $\mathbb{B}_1(0)$.
Note that 
 $$ |\partial_t \varphi_k | + |\nabla \varphi_k | + k|\Delta \varphi_k | \leq \frac{C}{k}.  $$
  The applying \eqref{abstest} with $v=\widecheck{\Heat}*h \geq \Phi^{2q'} \varphi_k^{2 q'}$ in $\R^n\times [0,\infty)$, where 
$h= |(-\partial_t -\Delta)(\Phi^{2q'} \varphi_k^{2 q'})|\upchi_{[0,\infty)}(t)$, and noting that
\begin{align*}
(-\partial_t -\Delta) v&=h\leq |(\partial_t +\Delta)(\Phi^{2q'}) \varphi_k^{2 q'}| \\
&\qquad + |(\partial_t +\Delta)(\varphi_k^{2 q'}) \Phi^{2q'}+  2 |\nabla (\Phi^{2q'}) \cdot \nabla (\varphi_k^{2 q'})|,
\end{align*}
 we find
\begin{align*}
 \int_{\R^n} \left [(\Phi \varphi_k)(\cdot,0)\right ]^{2b'} d u_0 
 & \lesssim \int_{0}^\infty\int_{\R^n} \left|(\partial_t  +\Delta)(\Phi^{2b'})\right|^{b'} \varphi_k^{2 b'}  \Phi^{2b'(1-b')} dx dt \\
 &\quad \quad +  \int_{0}^\infty\int_{\R^n} \left|(\partial_t  +\Delta)(\varphi_k^{2 b'}) \right|^{b'} \varphi_k^{2 b'(1-b')}  \Phi^{2b'} dx dt \\
 &\quad \quad +  \int_{0}^\infty\int_{\R^n} \left|\nabla (\varphi_k^{2 b'}) \right|^{b'} \varphi_k^{2b'(1-b')}  \Phi^{b'}|\nabla \Phi|^{b'} dx dt. 
\end{align*}
Here the constant $c$ depends only on $b$ and $\kappa$.
Thus, sending $k\rightarrow\infty$, obtain
\begin{align*}
& \int_{\R^n}  \Phi(\cdot, 0)^{2b'} d u_0 
  \lesssim \int_{0}^\infty\int_{\R^n} \left|(\partial_t  +\Delta)(\Phi^{2b'})\right|^{b'}   \Phi^{2b'(1-b')} dx dt \\
 & \qquad \lesssim \int_{0}^\infty\int_{\R^n}|\partial_t \Phi +\Delta \Phi|^{b'} \Phi^{b'} dx dt +   \int_{0}^\infty\int_{\R^n}   |\nabla \Phi|^{2q'} dx dt,
\end{align*}
 where we used the equality 
  $$(\partial_t +\Delta)(\Phi^{2b'})=  2b' \Phi^{2b'-1} (\partial_t \Phi  + \Delta \Phi)  + 2b'(2b'-1) \Phi^{2b'-2} |\nabla \Phi|^2.$$
At this point, applying Gagliardo--Nirenberg interpolation inequality (see \cite{N, FFRS}) in the form 
$$\int_{\R^n} |\nabla \Phi|^{2q'} dx \lesssim \left\| \Phi \right\|_{L^\infty(\R^n)}^{b'} \int_{\R^n} |D^2 \Phi|^{b'} dx, \qquad \forall t\in \R,   $$ 
and the assumption $0\leq \Phi\leq 1$, we obtain the claim  \eqref{trace-qprime}.

Inequality \eqref{trace-qprime} now implies that 
$$u_0(K) \lesssim\, {\rm cap}_{2,1, b'}(K\times \{0\})$$
for any compact set $K\subset\R^n$. Thus, in view of Theorem \ref{thm:cap-equiv} and \eqref{0-cap}, we obtain $\rm{(ii)}$.

$\bullet$ {\bf Proof of ${\rm (vi)} \Leftrightarrow {\rm (vii)}$:} The implication ${\rm (vii)} \Rightarrow {\rm (vi)}$ follows from Lemma \ref{Fchar}.
Conversely, by Lemma \ref{Fchar}, we see that 
 ${\rm (vi)}$ implies that $$u_0\in \TrMs{-\frac{2}{b}}{\frac{2}{b}}{b'}{b}{q}(\R^n)\subset B^{-\frac{2}{b-1}}_{\infty, \infty}(\R^n)$$ for any $0<q\leq \infty$.
 Moreover, the bound \eqref{I2b} implies that ${\bf I}_{\frac2b}[u_0]<+\infty$  almost everywhere. In particular,  
 $$\int_{\R^n} \frac{1}{|x_0-y|^{n-\frac2b}} d u_0(y)<+\infty$$ 
 for some $x_0\in\R^n$. Thus, $u_0$ vanishes at infinity, \textcolor{red}{i.e.,}   $u_0\in \mathcal{S}_0'(\R^n)$. Then by \cite[Proposition 4.6]{Bou}, $u_0$ has the representation \eqref{pre-u0}.
\end{proof}

\section{Semilinear parabolic equations with distributional initial data}

Let $1<r< (n+2)/2$ and $p>0$. Inspired by Theorem \ref{BP-Hung}, we define a space 
$$Y_{p, r}=\{u\in L^p_{\rm loc}(\R^{n+1}_{+}): \|u\|_{Y_{p, r}}<+\infty \},$$
where 
\begin{align*}
\|u\|_{Y_{p, r}}:= \sup_{E\subset \R^{n+1}_+} \left(\frac{ \iint_{E}|u(x,t)|^pdx dt}{{\rm cap}_{2,1, r}(E)}  \right)^{1/p},
\end{align*}
with the supremum being taken over all compact subsets $E$ of $\R^{n+1}_{+}=\R^n\times (0, \infty)$. 
Also, we define the space 
$$Z_{p, r}(\R^n):=\left\{f\in \mathcal{S}'(\R^n): \|f\|_{Z_{p,r}(\R^n)}:= \| e^{t \Delta} f\|_{Y_{p, r}}<+\infty \right\}.$$

For $b>1+\frac 2 n$, $Z_{b, b'}(\R^n)$ will be our space of initial data for the semilinear parabolic equation \eqref{ABC}.
To show its significance, we will now demonstrate that it  contains all of the spaces of initial data that were known  earlier in the literature for the existence of global in time solutions.
 First,  note  that by Theorem \ref{BP-Hung},  for $b>1+\frac 2 n$, we have 
\begin{align*}
Z_{b, b'}(\R^n)\cap \mathfrak{M}^+(\R^n)=\TrMs{-\frac{2}{b}}{\frac{2}{b}}{b'}{b}{q}(\R^n)\cap \mathfrak{M}^+(\R^n)
\end{align*}
for all $0<q\leq \infty$.  

For general signed tempered distributions, we have the following embedding results.
\begin{theorem}\label{BbFbZ}
Let $b>1+ \frac 2 n$. We have
\begin{equation}\label{abcran} 
\dot{B}^{-\frac 2 b}_{b}\mathsf{M}^{\frac{2}{b}, b'}_{b} \hookrightarrow \TrMs{-\frac{2}{b}}{\frac{2}{b}}{b'}{b}{b}(\R^n)  \hookrightarrow  Z_{b, b'}(\R^n).
\end{equation}
Similarly, for $0<\epsilon< n(b-1)/2-1$,
\begin{align*}
 \dot{B}^{-\frac 2 b}_{b}\mathcal{M}^{b(1+\epsilon), \frac{2(1+\epsilon)}{b-1}}  \hookrightarrow \dot{F}^{-\frac 2 b}_{b}\mathcal{M}^{b(1+\epsilon), 
\frac{2(1+\epsilon)}{b-1}}   \hookrightarrow  Z_{b, b'}(\R^n).
\end{align*} 
\end{theorem}
\begin{proof} By \eqref{MorreyEmb}, it suffices to prove only the second embedding in \eqref{abcran}. Let $f\in \TrMs{-\frac{2}{b}}{\frac{2}{b}}{b'}{b}{b}(\R^n)$. For any compact set $E\subset \R^n\times [0, \infty)$, we have 
\begin{align*}
\iint_{E} |e^{t\Delta} f(x)|^b dx dt \leq  \int_K \int_0^\infty |e^{t\Delta} f(x)|^b dt dx, 
\end{align*}
where $K\subset\R^n$ is such that $K\times\{0\}=\{(x,0): \exists\, t\geq 0 {\rm ~such ~ that ~} (x,t)\in E\}$ is the projection on $E$ onto the hyperplane $\{t=0\}$.
Thus, by Theorem \ref{temChar}, it follows that 
\begin{align*}
\iint_{E} |e^{t\Delta} f(x)|^b dx dt & \lesssim \|f\|^b_{\TrMs{-\frac{2}{b}}{\frac{2}{b}}{b'}{b}{b}(\R^n)} {\rm Cap}_{{\bf I}_{\frac{2}{b}}, b'}(K)\\
& \lesssim \|f\|^b_{\TrMs{-\frac{2}{b}}{\frac{2}{b}}{b'}{b}{b}(\R^n)} {\rm cap}_{2,1, b'}(K\times\{0\}) \\
& \lesssim \|f\|^b_{\TrMs{-\frac{2}{b}}{\frac{2}{b}}{b'}{b}{b}(\R^n)}   {\rm cap}_{2,1, b'}(E),
\end{align*}
where we used \eqref{0-cap} in the second inequality. In the last inequality, we used the fact that  ${\rm cap}_{ 2,1, b'}(K\times\{0\}) \lesssim {\rm cap}_{ 2,1, b'}(E)$, which follows from the  Lemma \ref{LipandCap} below.
This completes the proof of the theorem.
\end{proof}

\begin{lemma}\label{LipandCap}
Let $\beta\in (0, n+2)$ and $1<p<\frac{n+2}{\beta}$. Assume that $E\subset\R^{n+1}$ is a compact set and $\Phi: E\rightarrow \R^{n+1}$ is a Lipschitz map with Lipschitz constant $L$, i.e., 
$$|\Phi(z_1)-\Phi(z_2)|\leq L |z_1-z_2|$$
for all $z_1, z_2\in E$. Then we have
\begin{equation*}
{\rm Cap}_{\mathbb{I}_\beta, p}(\Phi(E))\leq A L^{n+2-\beta p}  {\rm Cap}_{\mathbb{I}_\beta, p}(E).
\end{equation*}
\end{lemma}
\begin{proof} The proof is based on the following dual definition of capacity:
$${\rm Cap}_{\mathbb{I}_\beta, p}(K)^{\frac 1 p} =\sup\{\mu(K): \mu \in \mathfrak{M}^+(K), \|\mathbb{I}_\beta[\mu]\|_{L^{p'}(\R^{n+1} )}\leq 1 \}$$
for any compact set $K\subset\R^{n+1}$, see \cite[Theorem 2.5.1]{AH}. We also need   the parabolic version of Wolff's inequality:
$$\|\mathbb{I}_\beta[\mu]\|_{L^{p'}(\R^{n+1})}^{p'} \simeq \int_{\R^{n+1}} \int_0^\infty \left( \frac{\mu(Q_r(x,t))}{r^{n+2-\beta p}} \right)^{\frac{1}{p-1}} \frac{dr}{r} d\mu(x,t),$$
see \cite[Theorem 2.7]{CO}. With these, the proof goes as in the proof  of \cite[Theorem 5.2.1]{AH}  that was presented for Bessel capacities (see also \cite[Corollary 5.2.3]{AH}.
\end{proof}

Note that Theorem \ref{BbFbZ} implies that 
$$\dot{B}^{-2/b}_{nb(b-1)/2, b} \hookrightarrow \dot{F}^{-2/b}_{nb(b-1)/2, b} \hookrightarrow Z_{b, b'}(\R^n), \quad b> 1+\frac 2 n,$$
which can be used to answer the question of Brandolese and Cortez in \cite{BCo}.

As another consequence of Theorem \ref{BbFbZ},  we have  the following result for locally finite  signed measures, which covers the case of Morrey spaces considered in \cite{Sou}. In fact, the case of 
signed measures has been considered in \cite{Q-H}.

\begin{corollary} \label{MorreyandCapspace} 
Let $b> 1+\frac{2}{n}$ and let $\mu \in \mathfrak{M}(\R^n)$ be such that 
$$|\mu|(K)\leq A\,  {\rm Cap}_{ {\bf I}_{\frac{2}{b}, b'}}(K)$$
for all compact sets $K\subset\R^n$.  Then we have 
$$ \|\mu\|_{Z_{b, b'}(\R^n)}\leq C \sup_{K\subset\R^n} \left\{ \frac{|\mu|(K)}{{\rm Cap}_{ {\bf I}_{\frac{2}{b}, b'}}(K)} \right\},$$ 
where the supremum is taken over all compact sets $K\subset\R^n$ such that ${\rm Cap}_{ {\bf I}_{\frac{2}{b}, b'}}(K)\not=0 $.
In particular, for any $1< q \leq \frac{n(b-1)}{2}$, we have 
\begin{equation}\label{M-Souplet}
\mathcal{M}^{q, \frac{2 q}{b-1}}(\R^n) \hookrightarrow Z_{b, b'}(\R^n).
\end{equation}
\end{corollary}
  
\begin{proof} By Theorem \ref{BP-Hung} (see also \eqref{FandR}), we have 
\begin{align*}
\|\mu\|_{Z_{b, b'}(\R^n)}\leq \| |\mu| \|_{Z_{b, b'}(\R^n)} \simeq \| |\mu| \|_{\TrMs{-\frac{2}{b}}{\frac{2}{b}}{b'}{b}{q}(\R^n)} \simeq \sup_{K\subset\R^n} \left\{ \frac{|\mu|(K)}{{\rm Cap}_{ {\bf I}_{\frac{2}{b}, b'}}(K)} \right\}.
\end{align*}

The embedding \eqref{M-Souplet} follows from the above inequality and Fefferman-Phong inequality. 
\end{proof}

We next show the relation of $Z_{b, b'}(\R^n)$ to certain Koch-Tataru type spaces.

\begin{theorem} \label{KT-FP}
Let $b>1+\frac 2 n$ and  $\epsilon \geq 0$.   Define the space 
$$T_{b, \epsilon}(\R^n):=\left\{f\in \mathcal{S}'(\R^n): \|f\|_{T_{b,\epsilon}(\R^n)}<+\infty \right\},$$
where 
$$\|f\|_{T_{b,\epsilon}(\R^n)}:= \left(\sup_{x\in\R^n,\,  r>0}  r^{\frac{2b(1+\epsilon)}{b-1}-n-2} \int_0^{r^2} \int_{B_r(x)} |e^{t\Delta} f|^{b(1+\epsilon)} dy dt \right)^{\frac{1}{b(1+\epsilon)}}.$$

Then we have 
$$T_{b,\epsilon}(\R^n) \hookrightarrow Z_{b, b'}(\R^n) \hookrightarrow T_{b,0}(\R^n), \qquad \epsilon>0.$$
\end{theorem}

\begin{proof} The embedding $Z_{b, b'}(\R^n) \hookrightarrow T_{b,0}(\R^n)$ follows from the fact that ${\rm cap}_{1,2, b'}(Q_r)=c\,  r^{n+2-2b'}, r>0$. On the other hand, for any $\epsilon\geq0$, we have
\begin{equation}\label{B-b-1}
\| e^{t\Delta} f \|_{L^\infty(\R^n)}\lesssim \| f\|_{T_{b,\epsilon}(\R^n)} t^{\frac{-1}{b-1}}, \qquad t>0.
\end{equation}
For this, see  Lemma \ref{ZtoBes} below. To prove the embedding $T_{b,\epsilon}(\R^n) \hookrightarrow Z_{b, b'}(\R^n)$, by H\"older's inequality, we may assume that $2b(1+\epsilon)/(b-1)\leq n+2$. 
With this, we now claim that 
\begin{align}\label{Claim-FP}
 \int_{t_0}^{t_0+r^2} \int_{B_r(x)} |e^{t\Delta} f|^{b(1+\epsilon)} dy dt \lesssim\|f\|_{T_{b,\epsilon}(\R^n)}^{b(1+\epsilon)} r^{n+2-\frac{2b(1+\epsilon)}{b-1}}
 \end{align}
for all $t_0 \geq 0$, $r>0$,  $x\in\R^n$, and $\epsilon\geq 0$.  Indeed, if $t_0\leq r^2$, then 
 \begin{align*}
 \int_{t_0}^{t_0+r^2} \int_{B_r(x)} |e^{t\Delta} f|^{b(1+\epsilon)} dy dt &\lesssim \int_{0}^{t_0+r^2} \int_{B_{\sqrt{t_0+r^2}}(x)} |e^{t\Delta} f|^{b(1+\epsilon)} dy dt\\
 &\lesssim \|f\|_{T_{b,\epsilon}(\R^n)}^{b(1+\epsilon)}  \left(\sqrt{t_0+r^2}\right )^{n+2-\frac{2b(1+\epsilon)}{b-1}}\\
 &\lesssim \|f\|_{T_{b,\epsilon}(\R^n)}^{b(1+\epsilon)}  r^{n+2-\frac{2b(1+\epsilon)}{b-1}}.
 \end{align*}

On the other hand, if $t_0>r^2$, then  by \eqref{B-b-1},
\begin{align*}
 \int_{t_0}^{t_0+r^2} \int_{B_r(x)} |e^{t\Delta} f|^{b(1+\epsilon)} dy dt &\lesssim  r^n \int_{t_0}^{t_0+r^2} \|e^{t\Delta}  f \|^{b(1+\epsilon)}_{L^\infty(\R^n)} dt\\
 & \lesssim  \|f\|_{T_{b,\epsilon}(\R^n)}^{b(1+\epsilon)} r^n \int_{t_0}^{t_0+r^2}  t^{-\frac{b(1+\epsilon)}{b-1}} dt\\
 &\lesssim \|f\|_{T_{b,\epsilon}(\R^n)}^{b(1+\epsilon)} r^{n+2-\frac{2b(1+\epsilon)}{b-1}}.
 \end{align*}

Thus we obtain the claim \eqref{Claim-FP} as desired.  Note that \eqref{Claim-FP}, with $\epsilon>0$, is a parabolic version of the Fefferman-Phong condition, by \cite[Theorem 4]{SW} we arrive at the embedding
 $T_{b,\epsilon}(\R^n) \hookrightarrow Z_{b, b'}(\R^n)$ provided $\epsilon>0$.
\end{proof}

In \cite{KY}, Kozono and Yamayaki used the Besov-Morrey space $\dot{B}_{\infty}^{\frac{n}{p}-\frac{2}{b-1}} \mathcal{M}^{q, \frac{n q }{p}}$, with certain indices $p$ and $q$, to obtain an existence result
for the Cauchy problem \eqref{ABC}. We now show that this space is also contained in $Z_{b, b'}(\R^n)$. 

\begin{corollary} \label{KYinZ}
Let $b>1+\frac{2}{n}$ and let $p, q$ be real numbers such that  $$b\leq q\leq p  {\rm ~and~} n(b-1)/2<p< n b(b-1)/2.$$ Then we have 
 \begin{equation*}
\dot{F}_{\infty}^{\frac{n}{p}-\frac{2}{b-1}} \mathcal{M}^{q, \frac{n q }{p}} \hookrightarrow  \dot{B}_{\infty}^{\frac{n}{p}-\frac{2}{b-1}} \mathcal{M}^{q, \frac{n q }{p}}  \hookrightarrow Z_{b, b'}(\R^n).
\end{equation*}
\end{corollary}

\begin{proof}
The first embedding is obvious. For the second embedding, note that for any $b<q\leq p$, we have
$$\mathcal{M}^{q, \frac{nq}{p}}\hookrightarrow \mathcal{M}^{b, \frac{nb}{p}}$$
by H\"older's inequality. Thus, it suffices to 
consider only the endpoint case $q=b$. 
 We start  with the smoothing property of the heat semigroup on Morrey spaces that can be found in \cite[Proposition 4.1]{Sou}: For any $1\leq p_1\leq p_2<\infty$, and $\lambda\in (0, n]$, we have 
 $$\|e^{t\Delta} f\|_{\mathcal{M}^{p_2, \lambda}} \lesssim t^{-\frac{\lambda}{2} (\frac{1}{p_1}-\frac{1}{p_2})}  \|e^{t\Delta} f\|_{\mathcal{M}^{p_1, \lambda}}, \qquad t>0.$$

Thus, for $b\leq p< n b(b-1)/2$ and $f\in \dot{B}_{\infty}^{\frac{n}{p}-\frac{2}{b-1}}\mathcal{M}^{b, \frac{n b}{p}}$,  if we denote 
$$N_f= \|f \|_{\dot{B}_{\infty}^{\frac{n}{p}-\frac{2}{b-1}}\mathcal{M}^{b, \frac{n b}{p}}},$$
then 
\begin{align*}
 \|e^{t\Delta} f\|_{\mathcal{M}^{b(1+\epsilon), \frac{nb}{p}}} & \lesssim t^{-\frac{n b}{2p} (\frac{1}{b}-\frac{1}{b(1+\epsilon)})}  \|e^{t\Delta} f\|_{\mathcal{M}^{b, \frac{n b}{p}}}\\
 & \lesssim   N_f \,  t^{-\frac{n b}{2p} (\frac{1}{b}-\frac{1}{b(1+\epsilon)})}  t^{-\frac{1}{b-1} + \frac{n }{2p}}\\
 & \lesssim  N_f  \, t^{-\frac{1}{b-1} + \frac{n }{2p(1+\epsilon)}},
 \end{align*}
 where in the second inequality  we used a characterization of Besov-Morrey space of  order  $s=n/p-2/(b-1)<0$ by the heat semiproup obtained in \cite[Proposition 2.22 and Equ. (2.57)]{Mazz}. 
 It follows that for any $x\in \R^n$ and $r>0$,
 \begin{align*}
 \int_{0}^{r^2} \int_{B_r(x)} |e^{t\Delta} f|^{b(1+\epsilon)} dy dt\lesssim \left(N_f\right)^{b(1+\epsilon)} r^{n-\frac{nb}{p}} \int_0^{r^2} t^{-\frac{b(1+\epsilon)}{b-1}+\frac{nb}{2p}} dt.
 \end{align*}

Note that, since  $p< n b(b-1)/2$, we have 
 $$-\frac{b(1+\epsilon)}{b-1}+\frac{nb}{2p}> -1,$$
 provided we choose 
 $$0<\epsilon < \left(1+\frac{nb}{2p} \right) \frac{b-1}{b} -1.$$
 
 Thus, for this choice of $\epsilon$,
 \begin{align*}
 \int_{0}^{r^2} \int_{B_r(x)} |e^{t\Delta} f|^{b(1+\epsilon)} dy ds &\lesssim\left(N_f\right)^{b(1+\epsilon)} r^{n-\frac{nb}{p}}  r^{-\frac{2b(1+\epsilon)}{b-1}+\frac{nb}{p} +2} \\
 &\lesssim \left(N_f\right)^{b(1+\epsilon)} r^{n+2-\frac{2b(1+\epsilon)}{b-1}},
 \end{align*}
 which by Theorem  \ref{KT-FP} yields the desired embedding.
\end{proof}

\begin{lemma}\label{ZtoBes} Let $b>1+ \frac 2 n$. We have $Z_{b, b'}(\R^n) \hookrightarrow \dot{B}^{-\frac{2}{b-1}}_{\infty, \infty}(\R^n)\cap \mathcal{S}'_0(\R^n)$ with the estimate
\begin{equation}\label{temBes}
\|e^{t\Delta} f \|_{L^\infty} \lesssim  t^{-\frac{1}{b-1}} \|f\|_{T_{b, 0}(\R^n)} \lesssim  t^{-\frac{1}{b-1}} \|f\|_{Z_{b, b'}(\R^n)}, \quad \forall \,t>0.
\end{equation}

In fact, for any $R>0$, we have 
$$\sup_{0<t<R} t^{\frac{1}{b-1}} \|e^{t\Delta} f \|_{L^\infty} \lesssim \|f\|_{T^{\sqrt{nR}}_{b, 0}(\R^n)},$$
where $T^{M}_{b,0}(\R^n), M>0,$ is a time truncated version of $T_{b,0}(\R^n)$  defined by
$$\|f\|_{T^{M}_{b,0}(\R^n)} := \left(\sup_{0<r<M,  \, x\in\R^n} r^{n+2-2b'} \int_{0}^{r^2} \int_{B_{r}(x)} |e^{s\Delta} f|^b dy ds  \right)^{\frac{1}{b}}.$$
\end{lemma}

\begin{proof}
Let $(x,t)\in \R^{n+1}_+$ be fixed. For any $\tau\in(0, t/2)$, we have $e^{t\Delta} f(x)=e^{(t-\tau)\Delta}\left[e^{\tau \Delta}f\right](x)$ and thus, by H\"older's inequality,
\begin{align*}
|e^{t\Delta} f(x)|
&=\left |\frac{2}{t}\int_0^{\frac{t}{2}} \int_{\R^n} \frac{e^{-\frac{|x-y|^2}{4(t-\tau)}}}{(4\pi (t-\tau))^{n/2}}      e^{\tau \Delta} f(y) dy d\tau \right|\\
 &\leq \left\{\frac{2}{t}\int_0^{\frac{t}{2}} \int_{\R^n} \frac{e^{-\frac{|x-y|^2}{4(t-\tau)}}}{(4\pi (t-\tau))^{n/2}}      |e^{\tau \Delta} f(y)|^{b} dy d\tau \right\}^{\frac{1}{b}}.
\end{align*}
Observe that if $\tau \in(0, t/2)$ and $\frac{x-y}{\sqrt{t}}\in k+[0,1]^n$ for some $k\in \mathbb{Z}^n$, then
\begin{equation}\label{Rntok}
\frac{e^{-\frac{|x-y|^2}{4(t-\tau)}}}{(4\pi (t-\tau))^{n/2}}    \leq \frac{e^{-\frac{|x-y|^2}{4t}}}{(4\pi (t-\tau))^{n/2}} \leq      \frac{1}{t^{n/2}} \frac{C(n, N)}{(1+ |k|)^{N}}, \qquad \forall N>n.
\end{equation}
Thus, letting $N=n+1$, we have 
\begin{align*}
\|e^{t\Delta} f\|_{L^\infty}
&\lesssim  \left \{ \sum_{k\in\mathbb{Z}^n}  \frac{t^{-\frac{n+2}{2}}}{(1+ |k|)^{n+1}}\int_0^{\frac{t}{2}} \int_{\frac{x-y}{\sqrt{t}} \in k +[0,1]^n } | e^{\tau \Delta} f(y)|^b dy d\tau \right\}^{\frac{1}{b}}\\
&\lesssim  \left \{ \sum_{k\in\mathbb{Z}^n}  \frac{t^{-\frac{n+2}{2}}}{(1+ |k|)^{n+1}} t^{(n+2- 2b')/2} \right\}^{\frac{1}{b}} \|f\|_{T^{\sqrt{nR}}_{b, 0}(\R^n)}\\
&\lesssim \, t^{-\frac{1}{b-1}} \|f\|_{T_{b, 0}(\R^n)}\\
&\lesssim \, t^{-\frac{1}{b-1}} \|f\|_{Z_{b, b'}(\R^n)}.
\end{align*}

In order to prove that  $f\in \mathcal{S}_0'(\R^n)$, we write, for any $\phi\in C_c^\infty(\R^n)$ and $t>0$,
\begin{align*}
\mathcal{F}^{-1}(\phi(t\cdot)\mathcal{F} f)&=\mathcal{F}^{-1}(\phi(t\cdot) e^{t^2|\cdot|^2} e^{-t^2|\cdot|^2} \mathcal{F} f)\\
&=\mathcal{F}^{-1}(\phi(t\cdot) e^{t^2|\cdot|^2})* \mathcal{F}^{-1}(e^{-t^2|\cdot|^2} \mathcal{F} f)\\
&=t^{-n}\mathcal{F}^{-1}(\phi(\cdot) e^{|\cdot|^2})(\cdot/t)* \mathcal{F}^{-1}(e^{-t^2|\cdot|^2} \mathcal{F} f).
\end{align*}
Thus, it follows from \eqref{temBes} that 
\begin{align*}
\|\mathcal{F}^{-1}(\phi(t\cdot)\mathcal{F} f)\|_{L^\infty}&=\|\mathcal{F}^{-1}(\phi(\cdot) e^{|\cdot|^2})\|_{L^1} \| \mathcal{F}^{-1}(e^{-t^2|\cdot|^2} \mathcal{F} f)\|_{L^\infty}\\
&=C(\phi) \|e^{t^2 \Delta} f\|_{L^\infty}\\
&\leq C(\phi) C \,  t^{-\frac{2}{b-1}} \|f\|_{Z_{b, b'}(\R^n)}\\
&\rightarrow 0,
\end{align*}
as $t\rightarrow\infty$. This implies that $f=\sum_{j\in\mathbb{Z}} \Dyadic f$ in $\mathcal{S}'(\R^n)$ and hence belongs to  $\mathcal{S}_0'(\R^n)$
 by Proposition \ref{BoudReal}.
\end{proof}

We now define a space 
$$X=\left\{ u\in L^b_{\rm loc}(\R^{n+1}_{+}): \| u\|_{X}<+\infty) \right\},$$
where 
$$\| u\|_{X}:= \sup_{t>0}t^{\frac{1}{b-1}} \|u(\cdot, t)\|_{L^\infty} + \|u\|_{Y_{b, b'}}.$$
The space $X$ will be our space of solutions to the Cauchy problem \eqref{ABC} with initial data $u_0\in Z_{b,b'}(\R^n)$.

\begin{lemma} \label{Tbound}
Let $b>1+\frac{2}{n}$ and define, for $u\in X$ and $(x,t)\in \R^{n+1}_+$, 
\begin{align*}
T(u)(x,t)&:=\mathsf{H}*(|u|^{b-1}u \,\upchi_{(0,\infty)}(\cdot))(x,t)\\
&= \int_{0}^t   e^{(t-\tau)\Delta}[|u(\tau)|^{b-1}u(\tau)] d\tau.    
\end{align*}
Then we have 
\begin{equation*}
\|T(u)\|_{X} \lesssim \|u\|^b_{X}, \qquad \forall  u\in X.
\end{equation*}
This result also holds if $|u|^{b-1} u$ is replaced by $|u|^{b}$ in the definition of $T$. Moreover, for any $R>0$,
 $$ \sup_{0<t<R} t^{\frac{1}{b-1}}\|T(u)(\cdot, t)\|_{L^\infty(\R^n)}\lesssim \left( \sup_{0<s<R}  s^{\frac{1}{b-1}}\|u(\cdot, s)\|_{L^\infty(\R^n)}^b +  \|u\|_{\widetilde{Y}_b^{\sqrt{nR}}} \right)^b,$$
 where we define
$$\|u\|_{\widetilde{Y}_b^{M}}:= \left(\sup_{0<r<M, \, x\in\R^n} r^{n+2-2b'}  \int_{0}^{r^2} \int_{B_r(x)} |u(y,s)|^b dy ds\right)^{\frac{1}{b}}, \quad M>0.$$

\end{lemma}

\begin{proof} We write, for $(x,t)\in \R^{n+1}_+$,
\begin{align*}
T(u)(x,t)&= \int_{0}^{\frac t 2 } \int_{\R^n} \frac{e^{-\frac{|x-y|^2}{4(t-\tau)}}}{(4\pi(t-\tau))^{\frac{n}{2}}}|u(y,\tau)|^{b-1} u(y,\tau) dy d\tau\\
&\quad + \int_{\frac t 2}^t \int_{\R^n} \frac{e^{-\frac{|x-y|^2}{4(t-\tau)}}}{(4\pi(t-\tau))^{\frac{n}{2}}}|u(y,\tau)|^{b-1} u(y,\tau) dy d\tau\\
&=: T_1(u)(x,t) + T_2(u)(x,t).
\end{align*}
By \eqref{Rntok},  we have 
\begin{equation}
\begin{aligned}\label{T1}
|T_1(u)(x,t)| &\lesssim   \sum_{k\in\mathbb{Z}^n}  \frac{t^{-\frac{n}{2}}}{(1+ |k|)^{n+1}}\int_0^{\frac{t}{2}} \int_{\frac{x-y}{\sqrt{t}} \in k +[0,1]^n } |u(y, \tau)|^b dy d\tau \\
&\lesssim   \sum_{k\in\mathbb{Z}^n}  \frac{t^{-\frac{n}{2}}}{(1+ |k|)^{n+1}} t^{(n+2- 2b')/2}   \\
&\lesssim \, t^{-\frac{1}{b-1}} \|u\|^b_{\widetilde{Y}_b^{\sqrt{nt}}}. 
\end{aligned}
\end{equation}

On the other hand, as the heat semigroup is a contraction on $L^p$ for $1\leq p\leq \infty$,
we find
\begin{align}\label{T2}
|T_2(u)(x,t)| &\lesssim  \int_{\frac t 2}^t \|u(\cdot, \tau)\|_{L^\infty}^b d\tau \\
&\lesssim   \int_{\frac t 2}^t  \tau^{-\frac{b}{b-1}} d\tau \,  \sup_{s>0}  s^{\frac{b}{b-1}}\|u(\cdot, s)\|_{L^\infty}^b  \nonumber\\
&\lesssim \, t^{-\frac{1}{b-1}} \,  \sup_{0<s<t}  s^{\frac{b}{b-1}}\|u(\cdot, s)\|_{L^\infty}^b. \nonumber
\end{align}
Combining estimates \eqref{T1},  \eqref{T2}, and the fact that ${\rm cap}_{2,1, b'}(Q_r)=c\,  r^{n+2-2b'}$, $r>0,$ we obtain, for all $t>0$, 
\begin{equation*}
t^{\frac{1}{b-1}}\|T(u)(\cdot, t)\|_{L^\infty} \lesssim \|u\|^b_{\widetilde{Y}_b^{\sqrt{nt}}} + \sup_{0<s<t}  s^{\frac{b}{b-1}}\|u(\cdot, s)\|_{L^\infty}^b   \lesssim \,  \|u\|^b_{X}.
\end{equation*}

Finally, we  notice  that since  
$$|T(u)(x,t)|\lesssim \mathbb{I}_2[|u|^b\upchi_{(0,\infty)}(\cdot)](x,t), \qquad (x,t)\in \R^{n+1}_+,$$
it follows from Proposition \ref{prop:equiv} and Theorem \ref{thm:cap-equiv} that 
\begin{equation*}
\|T(u)\|_{Y_{b,b'}}\lesssim \|u\|_{Y_{b,b'}}^b \lesssim  \|u\|^b_{X}.
\end{equation*}
\end{proof}

A consequence of Lemma \ref{Tbound} is the following local Lipschitz estimate for the operator $T$.

\begin{lemma}\label{Condiff} Let $b> 1+\frac2 n$ and $u_1, u_2\in X$. Then we have 
\begin{equation*}
\|T(u_1) -T(u_2)\|_{X} \leq C  \| u_1-u_2\|_{X} (\| u_1\|_{X} + \| u_2\|_{X})^{b-1}.
\end{equation*}
\end{lemma}

\begin{proof}  Using the inequality 
\begin{equation}\label{ABIn}
||A|^{b-1}A-|B|^{b-1}B|    \leq b|A-B|(|A|^{b-1}+|B|^{b-1}), \qquad \forall A, B \in \R,
\end{equation}
we have 
$$|T(u_1)-T(u_2)|\leq b\,  \mathsf{H}*[|u_1-u_2|(|u_1|^{b-1}  +|u_2|^{b-1})\upchi_{(0,\infty)(\cdot)}].$$

Then by Lemma \ref{Tbound}, we find 
\begin{align*}
\|T(u_1)-T(u_2)\|_{X} &\lesssim \| |u_1-u_2|^{\frac 1 b}(|u_1|^{b-1}  +|u_2|^{b-1})^\frac 1 b\|_{X}^b \\
& \lesssim \|u_1-u_2\|_{X} (\|u_1\|_{X} + \|u_2\|_{X})^{b-1}, 
\end{align*}
where we used H\"older's inequality in the last inequality.
\end{proof}

The following theorem provides our main existence result for signed initial data $u_0$ that belong to   $Z_{b, b'}(\R^n), b>1+\frac 2 n$.

 \begin{theorem}\label{thm:maintheorem}
Let $b>1+ \frac 2 n$ and $u_0\in \mathcal{S}'(\R^n)$. There exists $\varepsilon=\varepsilon(n,b)>0$ such that  if $ u_0\in Z_{{b, b'}}(\R^n)$ with
 $$\|u_0\|_{Z_{b, b'}(\R^n)} = \|e^{t\Delta} u_0\|_{Y_{b,b'}}\leq \varepsilon,$$ then there exists a unique small  global  solution $u\in X$ to the Cauchy problem \eqref{ABC},
 which satisfies the estimate
\begin{equation}\label{pwebyI}
|u(x,t)|\leq |e^{t\Delta}u_0(x)|+C\, \mathbb{I}_2(|e^{\tau \Delta}u_0|^b)(x,t)    
\end{equation}
for a.e. $(x,t)\in \R^n\times (0,\infty)$.
\end{theorem}

\begin{proof} 
For  $u_0\in Z_{b, b'}(\R^n)$ and $v\in X$, we define
\begin{align*}
S(v)&= e^{t\Delta} u_0 + \int_{0}^t e^{(t-\tau) \Delta} \left(|v(\cdot, \tau)|^{b-1}v(\cdot, \tau)\right) d\tau\\
&= e^{t\Delta} u_0 + T(v).
 \end{align*}

Since $b>1$, by Lemmas \ref{ZtoBes}, \ref{Tbound}, and \ref{Condiff}, there exist $\varepsilon_1=\varepsilon_1(n,b)>0$ and $\delta=\delta(n,b)>0$ such that for any $\varepsilon\in (0, \varepsilon_1]$, the map 
$$S: \overline{B}_\delta^X\rightarrow \overline{B}_\delta^X$$
is a contraction, where  $\overline{B}_\delta^X=\{v\in X: \|v\|_{X}\leq \delta\}$. Thus, by Banach's Fixed Point Theorem, there exists a unique solution $u\in \overline{B}_\delta^X$ to the equation
$u=S(u)$. Moreover, we have 
$$u=\lim_{N\rightarrow \infty} u_N \quad {\rm in} \quad X,$$ 
where the sequence of functions $\{u_N\}_{N\in \mathbb{N}}$ is defined  by  $u_1(t)=e^{t\Delta}u_0$ and 
\begin{align}\label{solapp}
u_{N+1}(t) &=e^{t\Delta}u_0   +\Heat \ast (|u_N|^{b-1}u_N\Chi{(0,\infty)}{\cdot}), \quad \forall N\geq 1.
\end{align}

It remains to show the pointwise estimate \eqref{pwebyI}. To this end, 
let $U(t)=e^{t\Delta}u_0= e^{t\Delta}u_0 \upchi_{(0,\infty)}(t)$. By Theorem \ref{thm:cap-equiv},  we have  
\begin{align*}
\int_{E}|U|^bdxdt \leq c_1 \varepsilon^b \PRcap{2}{b'}(E)  
\end{align*}
for all compact sets $E\subset \R^{n+1}$.
Then, by Proposition \ref{prop:equiv},
\begin{align}\label{pwiterep}
\PR{2}(\PR{2}(|U|^b))^b\leq c_2 \varepsilon^{b(b-1)}\PR{2}(|U|^b) <+\infty \quad {\rm a.e.}   
\end{align}

 We next  claim that  there exists a positive constant $\varepsilon_2\leq \varepsilon_1 $ such that for all $\varepsilon\in (0, \varepsilon_2]$, it holds that  
\begin{equation}\label{Claimpw}
|u_N|\leq |U| + C\, \PR{2}(|U|^b), \quad \forall N\geq 1.    
\end{equation}

Indeed,  $|u_1|\leq |U|$ and suppose that for some $N\geq 1$, we have 
\begin{equation*}
|u_N|\leq |U| + C_N\, \PR{2}(|U|^b).
\end{equation*}
Then it follows from \eqref{pwiterep} that 
\begin{align*}
|u_{N+1}| & \leq    |U| + c_3 \mathbb{I}_2[(C_N)^b \mathbb{I}_2(|U|^b)^b]  + c_3 \mathbb{I}_2[|U|^b]\\
&\leq  |U| + c_3 \mathbb{I}_2[|U|^b] \left\{  (C_N)^b c_2 \varepsilon^{b(b-1)}  +1 \right\},
\end{align*}
 where $c_3=c_3(n,b)$. Thus, 
\begin{align*}
|u_{N}|  \leq  |U|+ C_{N} \mathbb{I}_2[|U|^b], \quad \forall N\geq 1,
\end{align*}
 with $$C_1=0 {\rm ~and~} C_{N+1}= c_3 \left\{  (C_N)^b c_2 \varepsilon^{b(b-1)}  +1 \right\} {\rm ~for~} N\geq 1.$$
By induction, it is then easy to see that 
$C_N\leq c_3 b'$
for all  $N\geq 1$, provided  $\varepsilon \leq \varepsilon_2$, where $\varepsilon_2\leq \varepsilon_1$ is such that 
$$c_2 \varepsilon_2^{b(b-1)}\leq (c_3)^{-b} (b-1)^{b-1} b^{-b}.$$
This   proves the claim  \eqref{Claimpw} with $C=c_3 b'.$

To continue, we will now show that  there exists $c=c(n)>0$ such that
\begin{equation}\label{del}
|u_{N+1}-u_N|\leq c\, 2^{-N}    \PR{2}(|U|^b), \quad \forall N\geq 1.
\end{equation}

To see this, note that 
\begin{equation*}
|u_2-u_1|\leq a_1\PR{2}(|U|^b),
\end{equation*}
and by induction, suppose that 
\begin{equation}\label{diffN}
|u_{N}-u_{N-1}|\leq a_{N-1} \PR{2}(|U|^b)
\end{equation}
for some $N\geq 2$. Then using \eqref{solapp}, \eqref{diffN},  and  inequality \eqref{ABIn},
we bound
\begin{align*}
|u_{N+1}-u_{N}|& \leq C\, \PR{2}[||u_N|^{b-1}u_N-|u_{N-1}|^{b-1}u_{N-1}|]\\
&\leq C  \, \PR{2}[|u_N -u_{N-1}| (|u_N|^{b-1} + |u_{N-1}|^{b-1})]\\
&\leq C  a_{N-1} \, \PR{2} \left[ \PR{2}(|U|^b) (|u_N|^{b-1} + |u_{N-1}|^{b-1}) \right].\,
\end{align*}
where $C=C(n,b)$.
Combining this with \eqref{Claimpw}, we obtain 
\begin{align*}
|u_{N+1}-u_N|&\leq C a_{N-1} \PR{2}\left[\PR{2}(|U|^b) \left\{\PR{2} (|U|^b)^{b-1}+|U|^{b-1} \right  \} \right]\\
&\leq C a_{N-1}  \PR{2}[\PR{2}(|U|^b)^b] +  C a_{N-1}  \PR{2}\left[\PR{2}(|U|^b) |U|^{b-1}  \right].
\end{align*}
Note that, by Young's inequality,  
\begin{align*}
 (\PR{2}|U|^b)|U|^{b-1}\leq  c(\beta)(\PR{2}|U|^b)^b+\beta |U|^{b}, \quad \forall \beta>0,   
\end{align*}
which then yields 
\begin{align*}
|u_{N+1}-u_N| & \leq \{C a_{N-1}(1+c(\beta))  \} \PR{2}[\PR{2}(|U|^b)^b] +  C a_{N-1}  \beta  \PR{2}[|U|^b]\\
& \leq  \{C a_{N-1}(1+c(\beta))  \}  c_2 \varepsilon^{b(b-1)} \PR{2}[|U|^b] +  C a_{N-1}  \beta  \PR{2}[|U|^b],
\end{align*}
where we used \eqref{pwiterep} in the last bound. Hence,
\begin{align*}
|u_{N+1}-u_N|&\leq a_{N}\PR{2}[|U|^b],
\end{align*}
with $a_N\leq \frac 12 a_{N-1}$, provided we choose $\beta>0$ such that $C   \beta\leq 1/4$ and 
 $\varepsilon\in (0,\varepsilon_3]$, where $\varepsilon_3\leq \varepsilon_2$ is such that 
$$ \{C (1+c(\beta))  \}  c_2 \varepsilon_3^{b(b-1)}\leq 1/4.$$ 
This yields \eqref{del} with  $c=2 a_1$ as desired. At this point, using \eqref{del}, we have the pointwise limit
\begin{equation}\label{pointwiselimU}
u=\lim_{N\rightarrow \infty} u_N = u_1+ \sum_{N=1}^\infty(u_{N+1}-u_N),
\end{equation}
where the series   converges absolutely pointwise a.e. Finally, letting $N\rightarrow\infty$ in 
 \eqref{Claimpw}, we arrive at the pointwise estimate \eqref{pwebyI}.
\end{proof}

Our next goal is to show the regularity/persistency of solutions obtained in Theorem \ref{thm:maintheorem} when the initial datum $u_0$ belongs to 
$Z_{b,b'}(\R^n)\cap L^a(\R^n)$, $a>1$.  We begin with the following lemma.

\begin{lemma}\label{UnEst}
Let $b>1+\frac 2 n$ and $a>1$.  Then for any $u, v \in X\cap L^\infty_tL^a_x$, we have 
$$\|T(u)-T(v)\|_{L^\infty_tL^a_x} \leq C  \|u -v\|_{X}^{\frac{b-1}{b}} \| |u| + |v|\|_{X}^{\frac{(b-1)^2}{b}} \| |u| + |v| \|_{L^\infty_tL^a_x}.$$
\end{lemma}
\begin{proof}
By \eqref{ABIn}, we have 
\begin{equation}\label{Tuv}
|T(u)(x,t)-T(v)(x,t)|\leq C \int_0^t e^{(t-\tau)\Delta}h(u,v)(x) d\tau,  
\end{equation}
where $h(u,v)=|u-v|(|u|^{b-1} + |v|^{b-1})$.

Suppose first that $1+\frac{2}{n}<b\leq a$. Then from \eqref{Tuv} and  by Lemma \ref{Tbound}, we find
\begin{align*}
&\left\|[T(u)-T(v)](\cdot, t) \right\|_{L^a_x}\\
&\leq  \left\|[T(u)-T(v)](\cdot, t) \right\|_{L^\infty_x}^{\frac{b-1}{b}}  \left\|[T(u)-T(v)](\cdot, t) \right\|_{L^\frac{a}{b}_x}^{\frac{1}{b}} \\
&\lesssim t^{-\frac{1}{b}} \|h(u,v)^{\frac{1}{b}}\|_{X}^{b-1}  \left( \int_0^t  \| h(u,v)(\cdot, \tau) \|_{L^{\frac{a}{b}}_x} d\tau \right)^{\frac 1 b}\\
&\lesssim  \|u-v\|_{X}^{\frac{b-1}{b}}  \||u|+|v|\|_{X}^{\frac{(b-1)^2}{b}}   \sup_{\tau\in[0, \infty)}\| |u(\cdot, \tau)| + |v(\cdot, \tau)| \|_{L^a_x}.
\end{align*}

Thus we obtain the lemma in the case $b\leq a$. 

Suppose now that   $b> a>1$. We have
\begin{align*}
&\left\|[T(u)-T(v)](\cdot, t) \right\|_{L^a_x}\\
&\leq  \left\|[T(u)-T(v)](\cdot, t) \right\|_{L^\infty_x}^{\frac{a-1}{a}}  \left\|[T(u)-T(v)](\cdot, t) \right\|_{L^1_x}^{\frac{1}{a}} \\
&\lesssim t^{-\frac{1}{b-1}\frac{a-1}{a}} \|h(u,v)^{\frac{1}{b}}\|_{X}^{\frac{b(a-1)}{a}}  \left( \int_0^t  \| h(u,v)(\cdot, \tau) \|_{L^{1}_x} d\tau \right)^{\frac 1 a}.
\end{align*}

Also,  note that 
\begin{align*}
\| h(u,v)(\cdot, \tau) \|_{L^{1}_x} &\leq \| h(u,v)(\cdot, \tau)^{\frac{1}{b}} \|_{L^{\infty}_x}^{b-a} \| h(u,v)(\cdot, \tau)^{\frac{1}{b}} \|_{L^{a}_x}^{a}  \\
&\lesssim \| h(u,v)^{\frac 1 b} \|_{X}^{b-a} \tau^{-\frac{b-a}{b-1}} \sup_{\tau\in[0, \infty)}\| |u(\cdot, \tau)| + |v(\cdot, \tau)| \|_{L^a_x}^a.
\end{align*}


Thus, since $a>1$, we get 
\begin{align*}
&\left\|[T(u)-T(v)](\cdot, t) \right\|_{L^a_x}\\
&\lesssim t^{-\frac{1}{b-1}\frac{a-1}{a}} \| h(u,v)^{\frac 1 b} \|_{X}^{b-1} 
 \left( \int_0^t \tau^{-\frac{b-a}{b-1}} d\tau \right)^{\frac 1 a}   \| |u| + |v| \|_{L^\infty_tL^a_x} \\
 &\lesssim  \|u-v\|_{X}^{\frac{b-1}{b}}  \||u|+|v|\|_{X}^{\frac{(b-1)^2}{b}}  \| |u| + |v| \|_{L^\infty_tL^a_x}.
\end{align*}

This proves the lemma in the case $b>a$.
\end{proof}

%

\smallskip

We next prove a persistency result in $L^a(\R^n)$  for any $1<a<\infty$.

\begin{theorem}\label{PerRes}
 Let $b>1+\frac 2 n$ and $a>1$. Suppose that $u_0\in Z_{b,b'}(\R^n)\cap L^a(\R^n)$. Then the solution $u$ obtained in Theorem 
\ref{thm:maintheorem} also belongs to the space  $L^\infty((0,\infty), L^a(\R^n)) \cap C((0,\infty), L^a(\R^n))$.  Moreover, if $a\geq \frac{n(b-1)}{2}$, then $u\in C([0,\infty), L^a(\R^n))$, and thus by uniqueness it coincides wih the corresponding Weissler's solution. 
\end{theorem}
\begin{proof} Let $S, T$, $\varepsilon=\varepsilon(n,b)$,  $\delta=\delta(n,b)$, and $\{u_N\}_{N\in \mathbb{N}}$ be as in  the proof of Theorem \ref{thm:maintheorem}. Recall that $u_1=e^{t\Delta} u_0$.  Then $S:  \overline{B}^X_\delta \rightarrow  \overline{B}^X_\delta$ and there is constant $\kappa = \kappa(n, b)\in (0,1)$ such that 
$$\|T(v_1) - T(v_2)\|_X= \|S(v_1) - S(v_2)\|_X\leq \kappa \|v_1 -v_2\|_X$$
for all $v_1, v_2\in \overline{B}^X_\delta$.  Moreover,
$$\|u_{N+1}-u_N\|_X\leq C \kappa^N, \qquad \forall N\in\mathbb{N}.$$

On the other hand, it can be shown by Lemma \ref{UnEst} and induction that $\{u_N\}_{N\in\mathbb{N}} \subset L^\infty_t L^a_x$.
Thus, by Lemma \ref{UnEst}, we have, for any $N\geq 2$, 
\begin{align}\label{DiffinL}
\|u_{N+1}-u_N \|_{L^\infty_t L^a_x} &=\|T(u_{N})- T(u_{N-1}) \|_{L^\infty_t L^a_x}\\
&\leq C \kappa^{\frac{N(b-1)}{b}}   \||u_{N}| + |u_{N-1}| \|_{L^\infty_t L^a_x} \nonumber\\
&\leq C \kappa^{\frac{N(b-1)}{b}}    \max_{1\leq j\leq N} \|u_j\|_{L^\infty_t L^a_x},\nonumber
\end{align}
where we used that $\{u_N\}_{N\in \mathbb{N}}\subset \overline{B}^X_\delta$. By applying Lemma \ref{UnEst} with $u=u_1$ and $v=0$, inequality \eqref{DiffinL} also holds in the case $N=1$. For each integer $N\geq 1$, we now set
$$B_N=\max_{1\leq j\leq N} \|u_j\|_{L^\infty_t L^a_x}.$$

Then as $u_{N+1}= (u_{N+1} -u_{N}) +u_{N}$, we find 
$$\|u_{N+1}\|_{L^\infty_t L^a_x}\leq C \kappa^{\frac{N(b-1)}{b}} B_N + B_N=B_N(1+C \kappa^{\frac{N(b-1)}{b}}),$$
 which yields  that 
$$B_{N+1} \leq B_N(1+C \kappa^{\frac{N(b-1)}{b}}), \qquad N\geq 1.$$ 

Thus,
\begin{align*}
B_N &\leq B_{N-1} (1+C \kappa^{\frac{N(b-1)}{b}}) \\
& \leq  B_{N-2} (1+C \kappa^{\frac{N(b-1)}{b}})  (1+C \kappa^{\frac{(N-1)(b-1)}{b}})\\
&\leq B_1 \Pi_{j=1}^\infty  (1+ C \kappa^{\frac{j(b-1)}{b}})\\
&\leq C \|u_0\|_{L^a}.
\end{align*}

By \eqref{DiffinL}, this yields 
$$\|u_{N+1}-u_N \|_{L^\infty_t L^a_x} \leq C \kappa^{\frac{N(b-1)}{b}} \|u_0\|_{L^a}, \qquad \forall N\geq 1,$$
which implies that $u_N$ also converges to $u$ in $L^\infty_t L^a_x$.
At this point,  using \eqref{pointwiselimU}, Fatou's lemma and the above inequality, we arrive at 
\begin{equation*}
\|u\|_{L^\infty_t L^a_x}\leq C \|u_0\|_{L^a}.  
\end{equation*}

To show that 
$u \in C((0,\infty), L^a(\R^n))$,  it suffices to show that  $$T(u)\in C((0,\infty), L^a(\R^n))$$ for any $u\in X$. 
To this end, fix $t_0\in (0,\infty)$ and for any $t>t_0$, we 
 write 
\begin{align*}
T(u)(t)&=e^{(t-t_0)\Delta}T(u)(t_0)+ \int_{t_0}^{t} e^{(t-s)\Delta}(|u(s)|^{b-1} u(s)) ds \\
&=: A(t)+B(t).
\end{align*}
 
 Since $T(u)(t_0)\in L^a(\R^n)$ (by Lemma \ref{UnEst}) and the heat semigroup is  strongly continuous on $L^a$, we see that $A(t)\in C((t_0,\infty), L^a(\R^n))$.

To show the continuity of $B(t)$, we fix $\tau>t_0$. If $t>\tau$, then

\begin{align*}
B(t)-B(\tau)&= \int_{t_0}^{t} e^{(t-\tau)\Delta} e^{(\tau-s)\Delta}(|u(s)|^{b-1} u(s)) ds - B(\tau)\\
& =e^{(t-\tau)\Delta} B(\tau) - B(\tau). 
\end{align*}

Since $\sup_{s\geq t_0}\| u(s)\|_{L^\infty(\R^n)} \leq C t_0^{-\frac{1}{b-1}}$, we have  $$\sup_{t\geq t_0}\||u(s)|^b\|_{L^a} \leq C t_0^{-1} \sup_{t\geq t_0}\|u(s)\|_{L^a}\leq C t_0^{-1} \| u_0\|_{L^a}.$$ and $B(\tau)\in L^a$.  Thus,  
\begin{align*}
\|B(t)-B(\tau)\|_{L^a}= \|e^{(t-\tau)\Delta} B(\tau) - B(\tau)\|_{L^a}\rightarrow 0,
\end{align*}
as $t\rightarrow \tau+$, by the strong continuity of the heat semigroup on $L^a$. Now if $t_0<t<\tau$, we have 
\begin{align}\label{Bttau}
\|B(t)-B(\tau)\|_{L^a}&\leq \left \|  \int_{t_0}^t (e^{(\tau-s)\Delta} -e^{(t-s)\Delta}) (|u(s)|^{b-1} u(s)) ds \right \|_{L^a} \\
&\qquad + \left\|  \int_{t}^{\tau} e^{(\tau-s)\Delta} (|u(s)|^{b-1} u(s)) ds \right\|_{L^a} \nonumber\\
&\leq  \int_{t_0}^t   \left \| (e^{(\tau-s)\Delta} -e^{(t-s)\Delta}) (|u(s)|^{b-1} u(s)) \right \|_{L^a} ds \nonumber\\
&\qquad +  C  \sup_{s\geq t_0}\left\| |u(s)|^b \right\|_{L^a} (\tau-t). \nonumber
\end{align}

Given any  $\eta>0$,  choose a small $\gamma=\gamma(\eta, t_0)$  such that $0<\gamma<\tau-t_0$ and 
$$\gamma \sup_{s\geq t_0} \| |u(s)|^b\|_{L^a}\leq \eta/2.$$
As we will let $t\rightarrow \tau-$, we may assume that $\tau-\gamma<t<\tau$. Thus, we can write
\begin{align*}
&\int_{t_0}^t   \left \| (e^{(\tau-s)\Delta} -e^{(t-s)\Delta}) (|u(s)|^{b-1} u(s)) \right \|_{L^a} ds \\
&=\int_{t_0}^{\tau-\gamma}   \left \| (e^{(\tau-s)\Delta} -e^{(t-s)\Delta}) (|u(s)|^{b-1} u(s)) \right \|_{L^a} ds\\
&\qquad +\int_{\tau-\gamma}^t   \left \| (e^{(\tau-s)\Delta} -e^{(t-s)\Delta}) (|u(s)|^{b-1} u(s)) \right \|_{L^a} ds\\
&\leq \int_{t_0}^{\tau-\gamma}   \left \| (e^{(\tau-s)\Delta} -e^{(t-s)\Delta}) (|u(s)|^{b-1} u(s)) \right \|_{L^a} ds\\
&\qquad + 2 \gamma \sup_{s\geq t_0} \| |u(s)|^b\|_{L^a}. 
\end{align*}

Note that for each $s\in(t_0, \tau-\gamma)$, $\left \| (e^{(\tau-s)\Delta} -e^{(t-s)\Delta}) |u(s)|^{b-1} u(s) \right \|_{L^a}\rightarrow 0$ as $t\rightarrow \tau-$. Thus, by Lebesgue's Dominated Convergence Theorem, we find
\begin{align*}
\limsup_{t\rightarrow \tau-} \int_{t_0}^t   \left \| (e^{(\tau-s)\Delta} -e^{(t-s)\Delta}) |u(s)|^{b-1} u(s) \right \|_{L^a} ds \leq \eta.
\end{align*}

Then using this in \eqref{Bttau}, we arrive at
\begin{align*}
\limsup_{t\rightarrow \tau-} \|B(t)-B(\tau)\|_{L^a}\leq \eta.
\end{align*}

Since $\eta>0$ is arbitrary, we obtain that $B(t)\in C((t_0,\infty), L^a)$.

It remains to show the continuity of $u$ in $L^a$ at $t=0$ under the condition $a\geq \frac{n(b-1)}{2}$. For this, we need to show that 
\begin{equation}\label{limT}
\lim_{t\rightarrow 0+} \|T(u)(\cdot, t)\|_{L^a}=0.
\end{equation}

Let $\,1\leq p_1\leq p_2\leq \infty.$ We first recall   the standard heat semigroup estimate
\begin{equation*}
\|e^{\tau\Delta} f\|_{L^{p_2}} \leq C \tau^{-\frac{n}{2}(\frac{1}{p_1} -\frac{1}{p_2} ) } \|f\|_{L^{p_1}}, \quad \tau>0.
\end{equation*} 

Then if $a\geq b$, we have  
\begin{align}\label{Talb}
\|T(u)(\cdot, t)\|_{L^a}& \leq \int_{0}^t  \left\|e^{(t-s)\Delta}(|u(s)|^b)\right\|_{L^a} ds \\
& \leq C \int_{0}^t  (t-s)^{-\frac{n}{2}(\frac{b}{a} -\frac{1}{a} ) }  \| |u(s)|^b \|_{L^{a/b}} ds \nonumber\\
& \leq C \left( \sup_{0<s<t}\| u(s) \|_{L^a} \right)^b t^{1-\frac{n(b-1)}{2a}}. \nonumber
\end{align}

On the other hand, if $a<b$, then
\begin{align*}
\|T(u)(\cdot, t)\|_{L^a}& \leq \int_{0}^t  \left\|e^{(t-s)\Delta}(|u(s)|^b)\right\|_{L^a} ds \\
&\leq C \int_{0}^t  (t-s)^{-\frac{n}{2}(1 -\frac{1}{a} ) } \| |u(s)|^b \|_{L^{1}} ds\\
& \leq C \left(\sup_{\tau>0} \tau^{\frac{1}{b-1}} \|u(\tau)\|_{L^\infty} \right)^{b-a}  \left( \sup_{0<s<t}\| u(s) \|_{L^a} \right)^a\\
&\qquad \times \int_{0}^t  (t-s)^{-\frac{n}{2}(1 -\frac{1}{a} ) } s^{-\frac{b-a}{b-1}} ds.
\end{align*}

This gives
\begin{align}\label{Tanb}
\|T(u)(\cdot, t)\|_{L^a}  &\leq C \left(\|\tau^{\frac{1}{b-1}} u(\tau)\|_{L^\infty_\tau L^\infty_x} \right)^{b-a} \\
&\qquad \times \left( \sup_{0<s<t}\| u(s) \|_{L^a} \right)^a  t^{\left(1-\frac{1}{a}\right)\left (-\frac{n}{2}+\frac{a}{b-1}\right)}.\nonumber
\end{align}

Thus, it follows from \eqref{Talb} and \eqref{Tanb} that equality \eqref{limT} holds in the case $a>\frac{n(b-1)}{2}$.

Finally, we consider the case $a=\frac{n(b-1)}{2}$. Since $u_N\rightarrow u$ in $L^\infty_t L^a_x$, $u_N, u\in C((0,\infty), L^a)$, and $u_N(0)=u(0)=u_0$ for all $N\in\mathbb{N}$, it suffices to show that 
$u_N$ is continuous at $t=0$ for each $N\in\mathbb{N}$. Also, since $u_1=e^{t\Delta} u_0$ is continuous at zero, by induction it is enough to show that 
\begin{equation}\label{limTv}
\lim_{t\rightarrow 0+}\|T(v)(\cdot,t)\|_{L^a}=0
\end{equation}
for any $v \in L^\infty((0,\infty), L^a) \cap C([0,\infty), L^a)$ with $v(0)=u_0$ such that the function $\tau^{\frac{1}{b-1}}v(\tau)\in L^\infty((0,\infty), L^\infty(\R^n))$.  Indeed, given any $\eta>0$, choose a function $\phi\in C^\infty_c(\R^n)$ such that 
$$\|u_0-\phi\|_{L^a(\R^n)}\leq \eta$$
and then by the continuity of $v$ at $t=0$, there exists $s^{\ast}>0$ such that 
$$\sup_{0<s<s^{\ast}}\|v(s)-\phi\|_{L^a(\R^n)}\leq 2\eta.$$

We now write
\begin{align}\label{Tva}
\|T(v)(\cdot, t)\|_{L^a}& \leq C \int_{0}^t  \left\|e^{(t-s)\Delta}(|v(s)-\phi|^b)\right\|_{L^a} ds \\
&\qquad + C \int_{0}^t  \left\|e^{(t-s)\Delta}(|\phi|^b)\right\|_{L^a} ds \nonumber\\
&\leq C \int_{0}^t  \left\|e^{(t-s)\Delta}(|v(s)-\phi|^b)\right\|_{L^a} ds  + C t \| |\phi|^b \|_{L^a}. \nonumber
\end{align}

Note that as in \eqref{Talb} and \eqref{Tanb} with $a=\frac{n(b-1)}{2}$, we have,  for any $0<t<s^{\ast}$,  

\begin{align*}
\int_{0}^t  \left\|e^{(t-s)\Delta}(|v(s)-\phi|^b)\right\|_{L^a} ds &\leq C \left( \sup_{0<s<t}\| v(s)-\phi \|_{L^a} \right)^b\\
& \leq C (2\eta)^b
\end{align*}
if $a\geq b$, and 
\begin{align*}
&\int_{0}^t  \left\|e^{(t-s)\Delta}(|v(s)-\phi|^b)\right\|_{L^a} ds \\
& \leq C \left(\|\tau^{\frac{1}{b-1}} v(\tau)\|_{L^\infty_\tau L^\infty_x} \right)^{b-a}
\left( \sup_{0<s<t}\| v(s)-\phi \|_{L^a} \right)^a\\
& \leq C \left(\|\tau^{\frac{1}{b-1}} v(\tau)\|_{L^\infty_\tau L^\infty_x} \right)^{b-a} (2\eta)^a
\end{align*}
if $a<b$.

As $\eta$ is arbitrary, using these estimates in \eqref{Tva}, we see that  limit \eqref{limTv} holds as desired.
 \end{proof}

\vspace{.2in}
\noindent {\bf Acknowledgements.} G.Y.D. is supported in part by a start-up grant from the University of Toronto, 
and NSERC grants RGPIN-2018-06487 and RGPIN-2025-06419. N. C. Phuc is supported in part by  Simons Foundation [Record ID: MPS-TSM-00002686].

\vspace{.2in}

\noindent {\bf AI Usage Disclosure.} ChatGPT  Plus  GPT-5.6 Sol was used to show the continuity in time of $T(u)$ in the proof Theorem \ref{PerRes}. The authors  take full responsibility for the content of the paper.

\end{document}